\documentclass[10pt,oneside,english,final]{amsart}
\usepackage[T1]{fontenc}
\usepackage[latin9]{inputenc}
\usepackage{geometry}
\usepackage{array}
\usepackage{textcomp}
\usepackage{amsthm}
\usepackage{amsbsy}
\usepackage{graphicx}
\usepackage{amssymb}

\makeatletter

\providecommand{\tabularnewline}{\\}

\numberwithin{equation}{section} 
\numberwithin{figure}{section} 
\theoremstyle{plain}
\theoremstyle{plain}
\newtheorem{thm}{Theorem}
  \theoremstyle{definition}
  \newtheorem{defn}[thm]{Definition}
  \theoremstyle{plain}
  \newtheorem{lem}[thm]{Lemma}
  \theoremstyle{plain}
  \newtheorem{cor}[thm]{Corollary}
  \theoremstyle{plain}
  \newtheorem{prop}[thm]{Proposition}
  \theoremstyle{remark}
  \newtheorem{rem}[thm]{Remark}

\@addtoreset{equation}{section}

\@addtoreset{table}{section}

\@addtoreset{figure}{section}

\AtBeginDocument{

}

\usepackage{amsmath}
\usepackage{babel}

\makeatother

\makeatother

\usepackage{babel}

\begin{document}

\title{Partial Reset in Pulse-Coupled Oscillators}

\author{Christoph Kirst$^{1,2}$ and Marc Timme$^{1,3}$}
\begin{abstract}
Pulse-coupled threshold units serve as paradigmatic models for a wide
range of complex systems. When the state variable of a unit crosses
a threshold, the unit sends a pulse that is received by other units,
thereby mediating the interactions. At the same time, the state variable
of the sending unit is reset. Here we present and analyze a class
of pulse-coupled oscillators where the reset may be partial only and
is mediated by a partial reset function. Such a partial reset characterizes
intrinsic physical or biophysical features of a unit, e.g., resistive
coupling between dendrite and soma of compartmental neurons; at the
same time the description in terms of a partial reset makes possible
a rigorous mathematical investigation of the collective network dynamics.

The partial reset acts as a desynchronization mechanism. For $N$
all-to-all pulse-coupled oscillators an increase in the strength of
the partial reset causes a sequence of desynchronizing bifurcations
from the fully synchronous state via states with large clusters of
synchronized units through states with smaller clusters to complete
asynchrony. By considering inter- and intra-cluster stability we derive
sufficient and necessary conditions for the existence and stability
of cluster states on the partial reset function and on the local dynamics
of the oscillators, as specified by their rise function. This analysis
of invariance and stability  goes beyond monotonicity and curvature
of the rise function and classifies the local dynamics according to
certain expansion and contraction properties that depend also on the
third derivative of the rise function. We obtain analytical bounds
for the stability of cluster states and for a specific class of oscillators
a rigorous derivation of all $N-1$ bifurcation points. Thus the sequence
of bifurcation points is extensive (of the order of the system size
$N$) with the actual number of bifurcating states growing combinatorially
in $N$. We demonstrate that the entire sequence of bifurcations may
occur due to arbitrarily small changes of the reset function. We illustrate
that the transition is robust against structural perturbations and
prevails in presence of heterogeneous network connectivity and rise
functions with mixed curvature.
\end{abstract}
\maketitle

\section{Introduction}

\footnotetext[1]{Max Planck Institute for Dynamics and Self-organization (MPIDS), Bunsenstr. 10, 37073 Göttingen, Germany}
\footnotetext[2]{Bernstein Center for Computational Neuroscience (BCCN) Göttingen}
\footnotetext[3]{Fakultät für Physik, Georg-August-Universität Göttingen, Germany}

Networks of pulse-coupled units serve as paradigmatic models for a
wide range of physical and biological systems as different as cardiac
pacemaker tissue, plate tectonics in earthquakes, chirping crickets,
flashing fireflies and neurons in the brain \cite{Buck:1976,Buck:1988,OlamiFederChristensen:1992,BottaniDelamotte:1997}.
In such systems, units interact by sending and receiving pulses at
discrete times that interrupt the otherwise smooth time evolution.
These pulses may be sound signals, electric and electromagnetic activations
as well as packets of mechanically released stress. Pulses are generated
once the state of a unit crosses a certain threshold value (e.g. the
mechanical stress of a tectonic plate becomes sufficiently large or
the voltage across a nerve cell membrane becomes sufficiently high);
thereafter the state of the sending unit is reset.

Synchronization of oscillators is one of the most prevalent collective
dynamics in pulse-coupled systems \cite{Mirollo:1990,Ernst:1995:1570,Ernst:1998,Bressloff:1997a,BressloffCoombes:2000,HanselMatoMeunier:1995,Gerstner:1993,Gerstner:1996a,TimmeWolf:2008}.
Often not all units are synchronized but form clusters consisting
of synchronized sub-groups of units which in turn are phase-locked
to other clusters \cite{Ernst:1995:1570,TimmeWolfGeisel:2002b,MemmesheimerTimme:2005,Banaji:2002,HanselMatoMeunier:1995,HanselMatoMeunier:1993a,Liu:2001,Pikovsky:2001,Pinsky:1995}.

In neuronal networks synchronization and clustering of pulses constitute
potential mechanisms for effective feature binding. In this paradigm,
different information aspects of the same object represented by activity
of different nerve cells are pooled together by temporal correlations
and in particular due to synchronous firing \cite{Singer:1999a,SingerGray:1995}.
However, strong synchronized firing of nerve cells can also be detrimental:
synchrony is prominent during epileptic seizures \cite{EngelPedley:1997,McCormickContreras:2001}
and observed in the basal ganglia during Parkinson tremor \cite{ElbleKoller:1990}.
Here mechanisms for desynchronizing neural activity are desirable
\cite{Tass:2003,Maistrenko:2004}. 

To study key mechanisms that may underly synchronization, e.g. in
biological neural networks, analytical tractable models of pulse-coupled
oscillator are helpful tools \cite{Peskin:1984,Mirollo:1990,Kuramoto:1991,Abbott:1993,Ernst:1995:1570,Bressloff:1997a,TimmeWolfGeisel:2003a,Denker:2004}.
Here the rise of the state variable of a free oscillatory unit towards
the threshold, the unit's \emph{rise function} characterizes the \emph{sub-threshold
dynamics}. If after reception of a pulse the state variable of the
unit stays below threshold it is said to receive \emph{sub-threshold
input}, whereas excitation above the threshold is \emph{supra-threshold}.\emph{
}Mirollo and Strogatz \cite{Mirollo:1990} showed that biological
oscillators always synchronize their firing in homogeneous networks
with excitatory all-to-all coupling if the rise function has a concave
shape. The synchronization mechanism they find has two parts: (i)
effective decrease of phase differences of units due to sub-threshold
inputs and (ii) instant synchronization due to supra-threshold inputs
and subsequent reset to a fixed value. 

In general,\emph{ }supra-threshold\emph{ }excitation and a subsequent
reset is a dominant mechanism for synchronization of pulse-coupled
oscillators because input pulses that force non-synchronized units
to cross threshold at nearby times are reset to the same value leaving
the units in the same state or in very similar states afterwards.
Although this reset mechanism plays a crucial role in the synchronization
process and the coordination of pulse generation times, its implications
for the collective network dynamics has - to our knowledge - not been
investigated systematically so far: most existing model studies reset
the units with supra-threshold inputs to a fixed value independent
of the strength of supra-threshold excitation \cite{Bottani:1995,Ernst:1995:1570,Kuramoto:1991,SennUrbanczik:2001,TimmeWolfGeisel:2002b,Corral:1995b,Denker:2004,Hopfield:1995b}.
This results in a complete loss of information about the prior state
of the units and makes the dynamics non-invertible. Some other studies
consider the opposite extreme: a complete conservation of supra-threshold
inputs during pulse sending and reset \cite{Hopfield:1995b,Bressloff:1997a}.
Here we aim at closing this gap by presenting and analyzing a model
where the reset (and thus the loss of information about the prior
state and the strength of supra-threshold excitation) can be varied
systematically.

This article is organized as follows: In section \ref{cha:Pulse-Coupled-Oscillators-with-Partial-Reset}
we propose a simple model of pulse-coupled oscillators with \emph{partial
reset}, where the response to supra-threshold inputs can be continuously
tuned and is not an all-or-none effect. To isolate the consequences
of partial reset, we focus on homogeneous systems of all-to-all pulse-coupled
oscillators with convex rise function. In section \ref{cha:Network-Dynamics-with Partial Reset}
we briefly present the results of numerical simulations and find that
the partial reset has a strong influence on the collective network
dynamics: Increasing the strength of the partial reset induces bifurcations
from the fully synchronous state via states with large clusters of
synchronized units through states with smaller clusters to complete
asynchrony. We study this transition rigorously by considering existence
and stability periodic cluster states with respect to inter-cluster
properties in section \ref{sub: AsynchronousClusterStates} and to
intra-cluster properties in \ref{sub:Stability-against-Cluster}.
We derive conditions on the partial reset function and the rise function
to bound regions of existence and stability of cluster states. In
\ref{sub:Solvable Example} we present a rigorous derivation of $N-1$
bifurcation points for a specific class of oscillators. We demonstrate
that the entire sequence of bifurcations may occur for arbitrarily
small changes of the reset function, thus underlining the strong impact
of partial reset on collective network dynamics. In section \ref{sec:Robustness}
we numerically illustrate that the transition is robust against structural
perturbations and prevails in presence of heterogeneous network connectivity
and rise functions with mixed curvature. In section \ref{sec:Discussion}
we discuss our results and relate the simple partial reset model to
biophysically more realistic neuron models .

Specific aspects of the implications of linear partial reset on synchronization
properties for oscillators have been briefly reported before \cite{KirstGeiselTimme:2008}.

\section{\label{cha:Pulse-Coupled-Oscillators-with-Partial-Reset}Networks
of Pulse-Coupled Units with Partial Reset}

We first propose a class of pulse coupled threshold elements with
partial reset and thereafter focus on units that oscillate intrinsically.

\subsection{Absorption Rule and Instant Synchronization}

\begin{figure}
\begin{centering}
\includegraphics[scale=0.8]{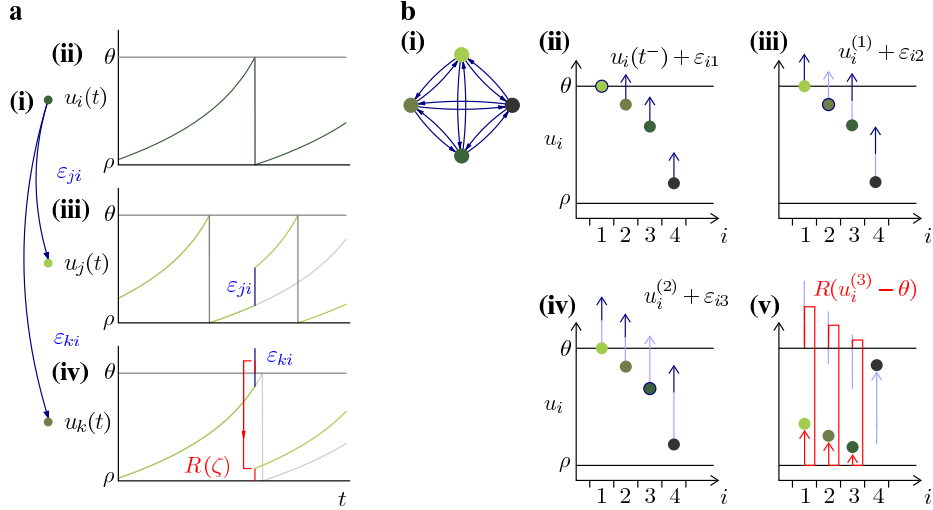} 
\par\end{centering}

\caption{\label{fig:Model-dynamics}Model dynamics. \textbf{(a)} Sample traces
of three units with \textbf{(i)} network connectivity $\varepsilon_{ji}=\varepsilon_{ki}>0$.
\textbf{(ii)} At time $t=t_{1}$ unit $i$ reaches the threshold $\theta$
and its membrane potential is reset to $\rho$. It generates a pulse
which is send to the units $j$ and $k$. \textbf{(iii)} unit $j$
receives the pulse and its membrane potential is increased to $u_{j}(t_{1}^{-})+\varepsilon_{ji}$,
the pulse is \emph{sub-threshold}. \textbf{(iv)} unit $k$ receives
a \emph{supra-threshold} pulse, $u_{k}(t_{1}^{-})+\varepsilon_{kj}\ge\theta$,
its membrane potential is set to $R(\zeta)=R\left(u_{j}(t_{1}^{-})+\varepsilon_{kj}-\theta\right)$
using the \emph{partial reset function} $R$. \textbf{(b)} Sample
avalanche process with $\Theta=\left\{ 1,2,3\right\} $ and $n=3$
in \textbf{(i)} a $N=4$ all-to-all network $\varepsilon_{ij}=(1-\delta_{ij})\varepsilon$.
\textbf{(ii)} unit $i=1$ reaches the threshold $\Theta^{(0)}=\left\{ 1\right\} $
and sends a pulse to the other units (arrows). Their potentials are
updated to $u_{i}^{(1)}=u_{i}(t^{-})+\varepsilon_{i1}$, causing unit
$i=2$ to generate a pulse, $\Theta^{(1)}=\left\{ 2\right\} $. \textbf{(iii)}
The pulse is received by the other units yielding a potential $u_{i}^{(2)}=u_{i}^{(1)}+\varepsilon_{i2}$
which brings unit $i=3$ above threshold $\Theta^{(2)}=\left\{ 3\right\} $.
\textbf{(iv)} The potentials become $u_{i}^{(3)}=u_{i}^{(2)}+\varepsilon_{i3}$
and no further unit crosses the threshold, $\Theta^{(3)}=\emptyset$.
\textbf{(v)} The avalanche stops and units that received supra-threshold
input are reset to $u_{i}(t)=\rho+R\left(u_{i}^{(3)}-\theta\right)$.}

\end{figure}

We consider $N$ threshold elements, which at time $t$ are characterized
by a single real state variable $u_{i}(t)$ with $i\in\{1,2,\dots,N\}$.
In the absence of interactions the state variables evolve freely according
to the differential equation \begin{equation}
\frac{d}{dt}u_{i}=F\left(u_{i}\right)\label{eq: Model free evolution}\end{equation}
with a smooth function $F:\mathbb{R}\rightarrow\mathbb{R}$ specifying
the intrinsic dynamics of the units. The free dynamics are endowed
with an additional nonlinear reset upon reaching a fixed threshold
$\theta$ from below\begin{equation}
u_{i}\left(t^{-}\right)=\theta\quad\Rightarrow\quad u_{i}\left(t\right)=\rho\label{eq: Model reset}\end{equation}
where $\rho<\theta$ is the reset value and we used the notation $u_{i}\left(t^{\pm}\right)=\lim_{s\searrow0}u_{i}\left(t\pm s\right)$.
By an appropriate shift and rescaling of the state variable and its
dynamics we set $\rho=0$ and $\theta=1$ without loss of generality.

The units are $\delta$-pulse coupled. If unit $j$ reaches the threshold
a pulse of strength $\varepsilon_{ij}\ge0$ is send instantaneously
to units $i$ and their membrane potential is increased by an amount
$\varepsilon_{ij}$\begin{equation}
u_{j}\left(t^{-}\right)=\theta\quad\Rightarrow\quad u_{i}^{(1)}=u_{i}\left(t^{-}\right)+\varepsilon_{ij}\label{eq: Model Pulse-Coupling}\end{equation}
If there is no connection from unit $j\rightarrow i$ we set $\varepsilon_{ij}=0$.

If a unit $i$ crosses the threshold due to a pulse from unit $j$\begin{equation}
u_{i}\left(t^{-}\right)+\varepsilon_{ij}\ge\theta\label{eq: Model Supra Threshold Input}\end{equation}
it is said to receive \emph{supra-threshold input}. If a unit receives
supra-threshold input it crosses the threshold from below, sends a
pulse and is reset. Previous models usually reset these units in the
same way as if they reached the threshold without this recurrent input.
In this type of reset, also referred to as the \emph{absorption rule}
(e.g. \cite{Mirollo:1990})

\begin{equation}
u_{i}\left(t\right)\ge\theta\quad\Rightarrow\quad u_{i}\left(t^{+}\right)=\rho\label{eq: Model Absorption}\end{equation}
the total supra.threshold input is lost. As a consequence two or more
units initially in different states $u_{i}$ and simultaneously receiving
supra-threshold inputs will all be reset to the same value $\rho$,
making the absorption rule a strong instant synchronizing element
of the network dynamics. An alternative considered in previous studies
as \cite{Kuramoto:1991,Hopfield:1995b} is \emph{total input conservation},
\begin{equation}
u_{i}\left(t\right)\ge\theta\quad\Rightarrow\quad u_{i}\left(t^{+}\right)=\rho+\left(u_{i}\left(t\right)-\theta\right)\label{eq: Model Total Charge Conservation}\end{equation}
i.e. the total \emph{supra-threshold input charge} $\zeta=u_{i}\left(t\right)-\theta$
is added to the potential $\rho$ after the reset.

\subsection{Partial Reset}

Here we propose a more general model where the reset value is given
by a \emph{partial reset function} $R(\zeta)$ that depends on the
supra-threshold input charge $\zeta=u_{i}(t)-\theta$, \begin{equation}
u_{i}(t)\ge\theta\quad\Rightarrow\quad u_{i}\left(t^{+}\right)=\rho+R\left(u_{i}(t)-\theta\right)\label{eq: Model Partial Reset}\end{equation}

We assume that supra-threshold inputs only lead to excitatory contributions
and thus define: 
\begin{defn}
A function $R:\mathbb{R}\rightarrow\mathbb{R}$ which is monotonically
increasing and satisfies $R\left(0\right)=0$ is called a \emph{partial
reset function}.
\end{defn}
For a \emph{linear partial response} we set\begin{equation}
R_{c}\left(\zeta\right)=c\zeta\label{eq: Model Rc}\end{equation}
with the remaining fraction $0\le c\le1$ of supra-threshold input
charge after the reset. For $c=0$ we recover the absorption rule
\eqref{eq: Model Absorption} while $c=1$ corresponds to total charge
conservation \eqref{eq: Model Total Charge Conservation}. 

Motivation for this extension comes from neural networks. Neurons
consist of functionally different compartments, including the dendrite
and the soma. While synaptic input currents are collected in the dendrite,
the electrical pulses are generated at the soma. Additional charges
not used to excite a spike may stay on the dendrite and contribute
to the membrane potential after reset at the soma. Due to intra-neuronal
interactions a part of this supra-threshold input charge may be lost
due to the reset at the soma:
\begin{defn}
A partial reset function $R$ is said to be \emph{neuronal}, if $0\le R(\zeta)\le\zeta$
for all $\zeta\ge0$. 
\end{defn}

\subsection{The Avalanche Process\label{sub: Model: The-Avalanche-Process}}

Since the interaction is instantaneous, a pulse generated by unit
$j$ may lift other units above threshold simultaneously. These then
generate a pulse on their own, etc. This leads to an avalanche of
pulses (cf. fig. \ref{fig:Model-dynamics}): Units reaching the threshold
at time $t$ due to the free time evolution define the triggering
set \begin{equation}
\Theta^{(0)}=\left\{ j\:|\: u_{j}\left(t^{-}\right)=\theta\right\} \label{eq: Model Avalanche triggering set}\end{equation}
The units $j\in\Theta^{(0)}$ generate spikes which are instantaneously
received by all the connected units $i$ in the network. In response,
their potentials are updated according to \begin{equation}
u_{i}^{(1)}:=u_{i}\left(t^{-}\right)+\sum_{j\in\Theta_{0}}\varepsilon_{ij}\label{eq: ui1}\end{equation}
The initial pulse may trigger certain other units $k\in\Theta^{(1)}=\left\{ k\:|\: u_{k}\left(t^{-}\right)<1\le u_{k}^{(1)}\right\} $
to spike, etc. This process continues $n\le N$ steps until no new
unit crosses the threshold. At each step $m\in\{2,3,\dots,n\}$ the
potentials are updated according to \begin{equation}
u_{i}^{(m+1)}:=u_{i}^{(m)}+\sum_{j\in\Theta_{m}}\varepsilon_{ij}\label{eq: uim}\end{equation}
 where \[
\Theta^{(m)}=\left\{ k\:|\: u_{k}^{(m-1)}<1\le u_{k}^{(m)}\right\} \]
 The potentials immediately after the avalanche $\Theta=\bigcup_{q=0}^{n}\Theta^{(q)}$
of size $a=\left|\Theta\right|$ are obtained via

\begin{equation}
u_{i}\left(t^{+}\right)=\begin{cases}
u_{i}\left(t^{-}\right)+\sum_{j\in\Theta}\varepsilon_{ij} & \, i\notin\Theta\\
\rho+R\left(u_{i}\left(t^{-}\right)+\sum_{j\in\Theta}\varepsilon_{ij}-\theta\right) & \, i\in\Theta\end{cases}\label{eq: Model avlanche partial reset}\end{equation}
 using the partial reset $R$ for units having received supra-threshold
inputs.

In general there is an ambiguity in fixing the precise order of potential
updates and resets during the avalanche. For example an instantaneous
reset could directly follow a single supra-threshold input pulse instead
of resetting the potentials at the end of the avalanche. Motivation
for our choice again comes form neuroscience: It is valid when the
time scale of the action potential (and subsequent reset) is much
faster than the time scale of the synaptic input currents. These in
turn should be much faster than the time scale of the mechanism reducing
the supra-threshold input, e.g. the refractory period (cf. the forthcoming
publication \cite{KirstTimme:2009}). Our model \eqref{eq: Model avlanche partial reset}
then is the limit where all these time scales become small compared
to the time scale of the intrinsic interaction free dynamics. 

For non-zero partial reset functions potential differences of oscillators
involved in a single avalanche will in general not be fully synchronized
after the reset. Thus despite the fact that units are generating pulses
simultaneously they can have different phases afterwards. We therefore
distinguish between \emph{phase synchrony} where units have identical
phases and the weaker condition of \emph{pulse synchrony} which corresponds
to simultaneous firing only but allows differences in the phases.
When examining the system with a higher time resolution phase synchronized
units will stay synchronized whereas pulse synchronized units fire
within a short time interval.

\subsection{Phase Representation of Pulse-Coupled Oscillators with Partial Reset\label{sec: Model Oscillatory-Units}}

In the remainder of this article we will concentrate on units with
strictly positive $F>0$ in \eqref{eq: Model free evolution}. Then
the individual units become oscillatory as the strictly monotonically
increasing trajectory $u_{i}(t)$ of a unit $i$ starting at $u_{i}(0)=0$
reaches the threshold after a time $T$ and is reset to zero again.
By an appropriate rescaling of time we set $T=1$. Defining a phase
like coordinate (cf. \cite{Mirollo:1990}) via\begin{equation}
\phi_{i}\left(t\right)=U^{-1}\left(u_{i}\left(t\right)\right):=\int_{0}^{u_{i}\left(t\right)}\frac{1}{F\left(u\right)}\mathrm{d}u\label{eq: Model Oscillators Transform}\end{equation}
the interaction free dynamics simplify to \begin{equation}
\frac{d}{dt}\phi_{i}(t)=1\label{eq: Model Oscillators dphi/dt eq one}\end{equation}
By definition $U^{-1}$ is strictly monotonically increasing and has
a strictly monotonically increasing inverse $U$. By our choice of
normalization they obey $U^{-1}\left(0\right)=0=U\left(0\right)$
and $U^{-1}\left(1\right)=1=U\left(1\right)$. Note that the function
$U\left(\phi\right)$ captures the intrinsic rise of the membrane
potential towards the threshold.
\begin{defn}
A smooth function $U:\left[0,\infty\right)\rightarrow\left[0,\infty\right)$
is called a \emph{rise function} if it is strictly monotonic increasing
$U'>0$ and is normalized to $U\left(0\right)=0$ and $U\left(1\right)=1$. 
\end{defn}

\begin{defn}
Given a rise function $U$ and a partial reset function $R$ we define
for $\varepsilon\ge0$ the \emph{(sub-threshold)} \emph{interaction
function} $H_{\varepsilon}:\left[0,\infty\right)\rightarrow\left[0,\infty\right)$
by \begin{equation}
H_{\varepsilon}\left(\phi\right):=H\left(\phi,\varepsilon\right):=U^{-1}\left(U\left(\phi\right)+\varepsilon\right)\label{eq: Model interaction function H}\end{equation}
and the \emph{supra-threshold interaction function} $J_{\varepsilon}:\left[U^{-1}\left(\theta-\varepsilon\right),\infty\right)\rightarrow\left[0,\infty\right)$
by \begin{equation}
J_{\varepsilon}\left(\phi\right):=J\left(\phi,\varepsilon\right):=U^{-1}\left(R\left(U\left(\phi\right)+\varepsilon-\theta\right)\right)\label{eq: Model def J}\end{equation}

\end{defn}
We remark that $H_{\varepsilon}^{-1}=H_{-\varepsilon}$. The pulse-coupling
in the potential representation eq. \eqref{eq: Model Pulse-Coupling}
carries over to the phase picture using the interaction function $H$:
\begin{equation}
\phi_{j}\left(t^{-}\right)=1\quad\Rightarrow\quad\phi_{i}\left(t+\tau\right)=H\left(\phi_{i}\left((t+\tau)^{-}\right),\varepsilon_{ij}\right)\label{eq: Model interaction}\end{equation}
Equation \eqref{eq: Model avlanche partial reset} for phases after
an avalanche $\Theta$ at time $t$ becomes\begin{equation}
\phi_{i}\left(t^{+}\right)=\begin{cases}
H\left(\phi_{i}\left(t^{-}\right),\sum_{j\in\Theta}\varepsilon_{ij}\right) & \quad i\notin\Theta\\
J\left(\phi_{i}\left(t^{-}\right),\sum_{j\in\Theta}\varepsilon_{ij}\right) & \quad i\in\Theta\end{cases}\label{eq: Model phase partial reset}\end{equation}

\section{Network Dynamics\label{cha:Network-Dynamics-with Partial Reset}}

To identify the effects of the partial reset on the collective network
dynamics we here first focus on homogeneous networks consisting of
$N$ units with all-to-all coupling and without self-interaction,
i.e.\begin{equation}
\varepsilon_{ij}=(1-\delta_{ij})\varepsilon\label{eq: All-to-All}\end{equation}
$i,j\in\left\{ 1,2,\dots,N\right\} $. We impose the condition $\sum_{j}\varepsilon_{ij}=(N-1)\varepsilon<\theta-\rho=1$
to avoid self-sustained avalanches of infinite size.

The analysis of Mirollo and Strogatz \cite{Mirollo:1990} shows that
in this situation (with a slightly different avalanche process) synchronization
from almost all initial conditions is achieved when the rise function
is concave ($U''<0$) and the absorption rule ($R\equiv0$) is used.
In fact, their result can be generalized to the partial reset model
used here and any partial reset function $R$ that is non-expansive
(e.g. $R'\le1$). For expansive $R$ the synchronized state no longer
has to be the global attractor of the dynamics and typically irregular
dynamics are observed. The proof of synchronization for non-expansive
partial resets and the analysis of the bifurcation to the irregular
dynamics are presented in a forth coming study \cite{KirstTimme:2009}. 

In this article we concentrate on convex rise functions $U$, i.e.\begin{equation}
\frac{d^{2}}{d\phi^{2}}U\left(\phi\right)>0\:.\label{eq: Analyse U convex}\end{equation}
This property holds for large class of conductance based leaky-integrate-and-fire
(LIF) neurons and a class of quadratic-integrate-and-fire (QIF) neurons
(cf. appendix \eqref{sec:Rise-Functions}). Studying convex rise functions
is further motivated by the fact that for these rise functions we
already observe a rich diversity of collective network dynamics with
a strong dependence on the partial reset $R$. However, some of our
results also apply to more general rise functions and in particular
to sigmoidal shapes as often found for neurons \cite{Ermentrout:1996,FourcaudTrocmeHanselVreeswijkBrunel:2003}.

\subsection{\label{sec: Analyse Numerics}Numerical Results: A Sequence of Desynchronizing
Bifurcations }

Systematic numerical investigations indicate a strong dependence of
the network dynamics on the partial reset $R$. In particular, we
find synchronous states, cluster states, asynchronous states and a
sequential desynchronization of clusters when increasing the partial
reset strength, e.g. by increasing the parameter $c$ when using $R=R_{c}$. 

Starting in the synchronized state and then applying a small perturbation
to the phases we observe that the synchronized state is stable for
sufficiently small $c$ (fig. \ref{fig: Desync syncIC} I). When increasing
the partial reset strength the synchronized state becomes unstable
and we observe smaller clusters in the asymptotic network dynamics
(fig. \ref{fig: Desync syncIC} II) where the final cluster state
depends on the precise form of the perturbation. The maximally observed
cluster sizes depend on the value of $c$ (fig. \ref{fig: Desync prob}).
For sufficiently large $c$ only the asynchronous splay state, i.e.
a state with maximal cluster size $a=1$ , is observed (fig. \ref{fig: Desync syncIC}
III). 

Starting from random initial conditions we find that for sufficiently
small $c$ the synchronous state coexists with a variety of cluster
states and the asynchronous state (fig. \ref{fig: Desync randomIC}
I and fig. \ref{fig: Desync prob}). Increasing $c$, the states involving
larger clusters become unstable (fig. \ref{fig: Desync randomIC}
II) until finally all random initial conditions lead to the asynchronous
state (fig. \ref{fig: Desync randomIC} III). 

What is the origin of this rich repertoire of dynamics and which mechanisms
control the observed transition of sequential desynchronizing bifurcations?
To answer these questions, we analytically investigate the existence
and stability of periodic states involving clusters of arbitrary sizes
$a\le N$. The following analysis reveals that the sequence of bifurcations
is controlled by two effects: sub-threshold inputs that are always
synchronizing and supra-threshold inputs that are either synchronizing
or desynchronizing depending on the partial reset strength.

\begin{figure}
\begin{centering}
\includegraphics[scale=0.72]{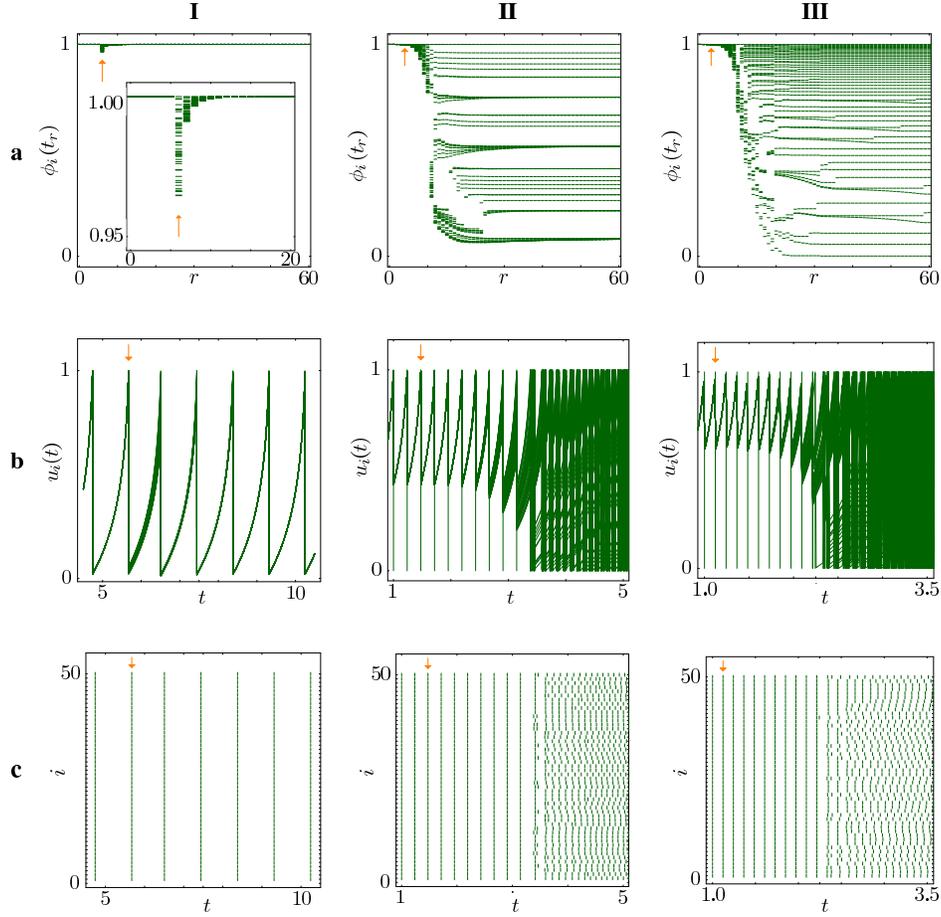} 
\par\end{centering}

\caption{\label{fig: Desync syncIC}Desynchronization transition in a network
with parameters $N=50$, $\varepsilon=0.0175$, $U=U_{b}$, $b=-3$,
$R=R_{c}$ for different values of the partial reset strength: \textbf{(I)}
$c=0.025<c_{\mathrm{cr}}^{(N)}$, \textbf{(II)} $c=0.5\in\left(c_{\mathrm{cr}}^{(N)},c_{\mathrm{cr}}^{(2)}\right)$
and \textbf{(III)} $c=0.7>c_{\mathrm{cr}}^{(2)}$. Plotted are \textbf{(a)}
the phases $\phi_{i}\left(t_{r}\right)$ of all units at pulse generation
times $t_{r}$ of the $r$-th spike of a reference unit (cf. return
map \eqref{eq: Model return map general}), \textbf{(b)} potential
traces of all units and \textbf{(c)} the raster plots marking the
times of pulse generation of each unit $i$. The network is started
in the synchronous state and then a small perturbation is applied
at a time indicated by arrows. }

\end{figure}

\begin{figure}
\begin{centering}
\includegraphics[scale=0.72]{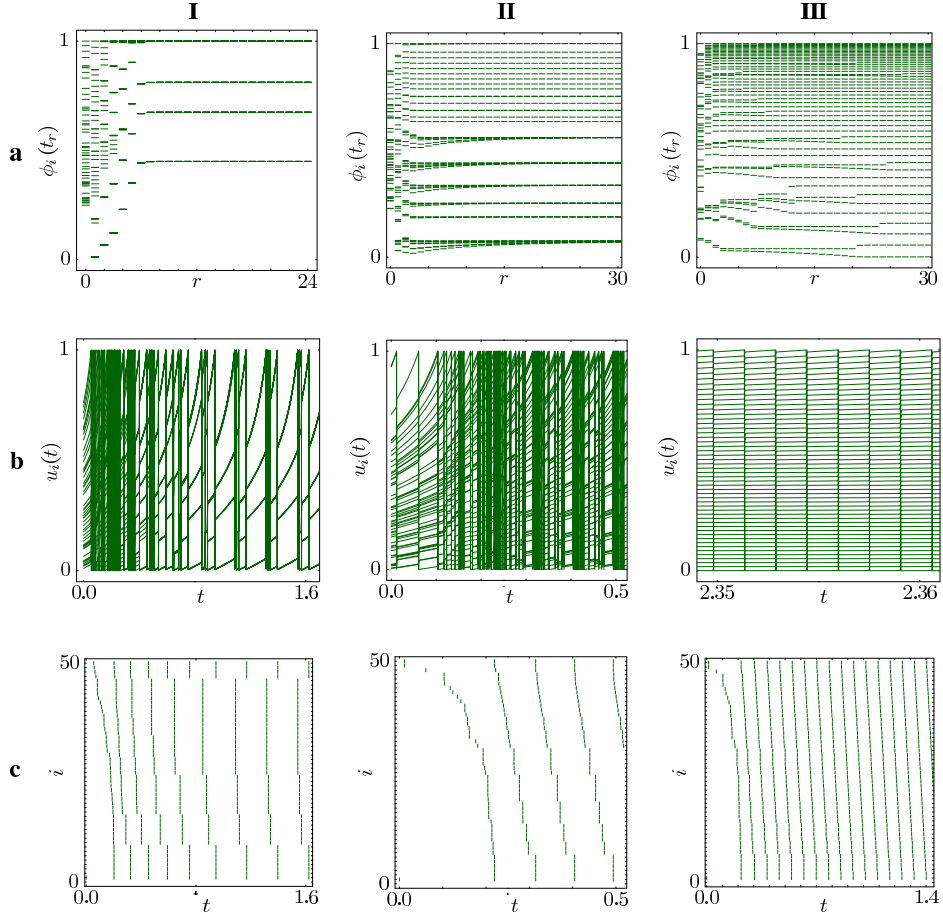} 
\par\end{centering}

\caption{\label{fig: Desync randomIC}Same as in Fig. \ref{fig: Desync syncIC}
but starting the dynamics form random initial phases distributed uniformly
in the interval $[0,1)$. \textbf{(IIIb)} As the transient is similar
to (IIb) a close up of the asymptotic \emph{asynchronous} (splay)
state is shown instead. In (c)\textbf{ }the oscillator are labeled
by increasing initial phases.}

\end{figure}

\begin{figure}
\begin{centering}
\includegraphics{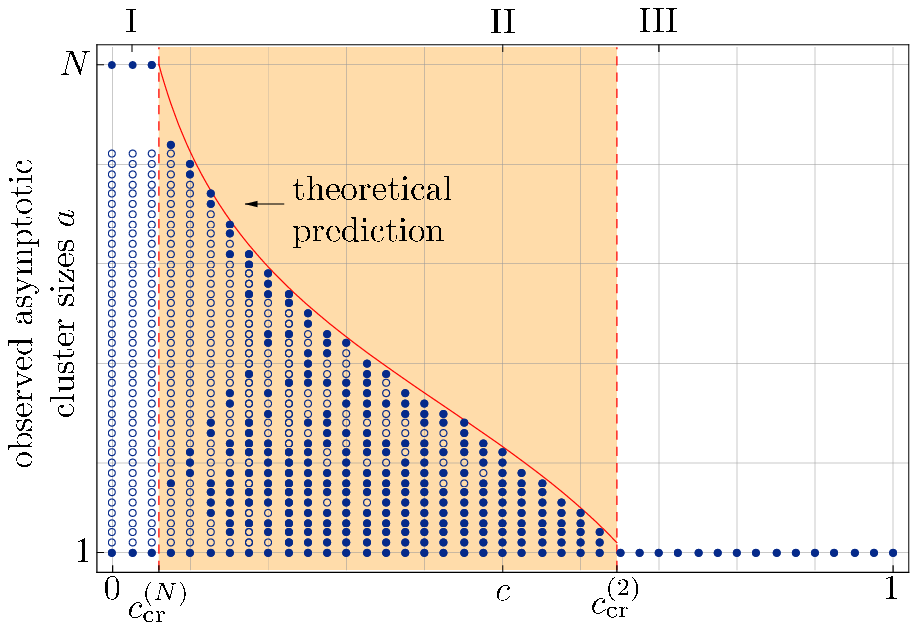} 
\par\end{centering}

\caption{\label{fig: Desync prob}Sequential desynchronization transition in
a network of $N=50$ units ($U=U_{b}$, $b=-3$, $\tilde{\varepsilon}=0.0175$).
Asymptotic cluster sizes $a$ observed in the asymptotic network dynamics
of $500$ simulations for each $c\in\left\{ 0,0.025,\dots,1\right\} $
starting from a perturbed synchronous state ($\bullet$, cf. fig.
\ref{fig: Desync syncIC}) or from other random phases distributed
uniformly in $[0,1)$ ($\circ$ cf. fig. \ref{fig: Desync randomIC}).
The shaded area marks the sequential desynchronizing transition, the
solid line shows the exact theoretical prediction \eqref{eq: Analyse c crit implicit theorem}
continuously interpolating the $N-1$ bifurcation points $c_{\mathrm{cr}}^{(a)}$,
$a\in\left\{ 2,3,\dots,N\right\} $ above which clusters of size $a$
become unstable. The gap for clusters sizes $43<a<50$ at $0<c<c_{\mathrm{cr}}^{(2)}$
appears as cluster states involving these avalanche sizes do not exist
according to lemma \ref{lem: existence asynchronous state}. }

\end{figure}

\subsection{\label{sub:Strategy}Strategy of the Analysis}

We split up our analysis of the dynamics into two parts. First we
assume that all avalanches are invariant, i.e. the given clusters
do not decay into smaller sub-clusters. This assumption allows to
group all oscillators firing in a single avalanche together into a
single {}``meta-oscillator'' with increased firing strength and
an effective self interaction (cf. fig. \ref{cap: Network reduction}).
The analysis of a state $\Phi$ in homogeneous all-to-all network
($\varepsilon_{ij}=(1-\delta_{ij})\varepsilon$) of $N$ oscillators
with avalanche sizes $a_{s}$, $s\in\left\{ 1,2,\dots,m\right\} $,
$\sum_{s}a_{s}=N$ then reduces to analyzing a network of $m$ meta-oscillators
with coupling strengths \begin{equation}
\varepsilon_{ij}=(1-\delta_{ij})\varepsilon_{i}+\delta_{ij}\varepsilon_{ii}\label{eq: effective coupling}\end{equation}
and $\varepsilon_{i}=a_{i}\varepsilon$, $\varepsilon_{ii}=(a_{i}-1)\varepsilon$.

In a second step we derive conditions under which an avalanche of
a certain size will indeed not decay into smaller groups.

\begin{figure}
\begin{centering}
\includegraphics{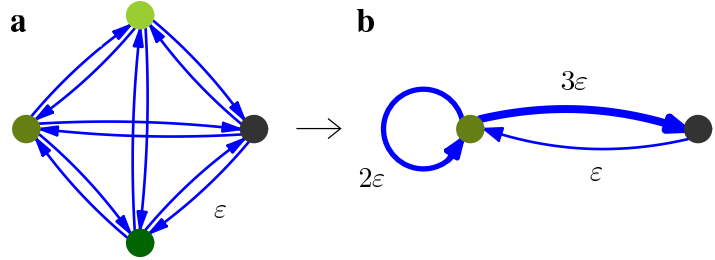}
\par\end{centering}

\caption{\label{cap: Network reduction}Network reduction method. \textbf{(a)}
A network of $N=4$ oscillators producing an avalanche of three units
as in Fig. \ref{fig:Model-dynamics} may be effectively reduced to
\textbf{(b)} a smaller network of size $N=2$ with inhomogeneous coupling
and self-interactions by grouping the oscillators involved in the
avalanche to a single meta-oscillator. The effective network connectivity
then has the form \eqref{eq: effective coupling}.}

\end{figure}

\subsection{Notations: State Space, Firing and Return Map}

A state of a network of $N$ pulse-coupled oscillators is completely
specified by a phase vector \begin{equation}
\Phi=\left(\phi_{1},\phi_{2},\dots,\phi_{N}\right)\in S^{N}=\underbrace{S^{1}\times\dots\times S^{1}}_{N\:\text{times}}\label{eq: Model Phase vector}\end{equation}
where $\phi_{i}\in S^{1}=\mathbb{R}/\mathbb{Z}$ are the phases of
the individual units. Since the time evolution in between avalanches
is a pure phase shift \eqref{eq: Model Oscillators dphi/dt eq one}
it is convenient to consider a Poincare section $\mathcal{S}$ of
$S^{N}$ with states just before the firing of one or more oscillators,
i.e.\begin{equation}
\mathcal{S}=\left\{ \Phi\in S^{N}\:|\:\exists j\in\left\{ 1,\dots,N\right\} ,\:\phi_{j}=1\right\} \:.\label{eq: Model state space}\end{equation}
It is convenient to relabel the oscillators after each avalanche such
that $1=\phi_{1}\ge\phi_{2}\ge\dots\phi_{N-1}\ge\phi_{N}>0$. To specify
the state of the network completely the permutation $\pi^{-1}$ used
for relabeling of the oscillators is remembered. The largest phase
$\phi_{1}=1$ thus belongs to the oscillator $i=\pi(1)$, the second
largest $\phi_{2}$ to $i=\pi(2)$, etc., and the smallest $\phi_{N}$
to oscillator $i=\pi(N)$. Thus an equivalent description of the state
space $\mathcal{S}$ is given by \begin{equation}
\mathcal{S}^{p}=\left\{ \left(\left(\phi_{2},\dots,\phi_{N}\right),\pi\right)\in S^{N-1}\times S_{N}\:|\:1\ge\phi_{2}\ge\dots\phi_{N-1}\ge\phi_{N}\ge0\right\} \label{eq: Model state space permutation}\end{equation}
Here $S_{N}$ is the group of all permutations of $N$ elements. We
use the convention that all index labels $i$ are taken modulo the
network size $N$, e.g. labels $i=j$ and $i=N+j$ denote the same
oscillator. 

The Poincare map of the network dynamics for the Poincare section
$\mathcal{S}$ is the \emph{firing-map} $\mathbf{K}$ that maps the
state $\Phi\in\mathcal{S}$ of the network just before the $s$-th
firing time $t_{s}$ of an avalanche to the state just before the
next avalanche that occurs at time $t_{s+1}$:\begin{equation}
\mathbf{K}\left(\Phi\left(t_{s}^{-}\right)\right)=\Phi\left(t_{s+1}^{-}\right)\in\mathcal{S}\label{eq: Model firing map general}\end{equation}
Having determined the next avalanche $\Theta$ from a state $\Phi\in\mathcal{S}$,
the map $\mathbf{K}$ is a composition of the avalanche map \eqref{eq: Model phase partial reset}
and a subsequent shift of all phases to a state in $\mathcal{S}$.
Note that the firing map is fully determined by the pair $\left(\Theta,\sigma\right)$
which is a function of $\Phi$. We denote a phase shift of size $\sigma$
by

\begin{equation}
S\left(\phi,\sigma\right):=S_{\sigma}\left(\phi\right):=\phi+\sigma\label{eq: Model def S}\end{equation}
The equivalent firing map acting on the state space $\mathcal{S}^{p}$
is denoted by $\mathbf{K}^{p}$. For the phase part we write \[
\mathbf{K}_{\Phi}^{p}\left(\left(\psi_{2}^{(0)},\dots,\psi_{N}^{(0)}\right),\pi^{(0)}\right)=\left(\psi_{2}^{(1)},\dots,\psi_{N}^{(1)}\right)\,.\]
To track the network dynamics we consider a mapping of the state just
before a fixed reference oscillator $k$ fires in an avalanche at
time $t_{r}$ to the state just before this oscillator fires again
at $t_{r+1}$: \begin{equation}
\mathbf{M}\left(\Phi\left(t_{r}^{-}\right)\right)=\Phi\left(t_{r+1}^{-}\right)\label{eq: Model return map general}\end{equation}
$\mathbf{M}$ is called the \emph{return-map} and is the Poincaré
map of the system on the section $\left\{ \Phi\in\mathcal{S}\,|\,\phi_{k}=1\right\} $.
Again the equivalent return map acting on $\mathcal{S}^{p}$ is denoted
by $\mathbf{M}^{p}$. The number $m$ of avalanches occurring in the
application of the return map is a function of the initial phase vector
$\Phi=\Phi\left(t_{r}^{-}\right)$ and the return map $\mathbf{M}$
then is a composition of $m$ firing maps $\mathbf{K}$. Thus $\mathbf{M}$
is completely specified by an ordered \emph{firing sequence} \begin{equation}
\mathcal{F}=\mathcal{F}\left(\Phi\right)=\left\{ \left(\Theta_{s},\sigma_{s}\right)\right\} _{s=1}^{m}\label{eq: Model F firing sequence}\end{equation}
where the pairs $\left(\Theta_{s},\sigma_{s}\right)$ specify the
avalanche set $\Theta_{s}$ and subsequent shift $\sigma_{s}$ of
the $s$-th firing map. 

Given a firing sequence \eqref{eq: Model F firing sequence}, we set
$a_{s}=\left|\Theta_{s}\right|$ and in the case of homogenous networks
with coupling \eqref{eq: All-to-All}, $\varepsilon_{s}=a_{s}\varepsilon$.
A composition of shift and interaction maps is denoted as\begin{equation}
\bigodot_{s=1}^{m}\left(S_{\sigma_{s}}\circ H_{\varepsilon_{s}}\right)\left(\phi\right):=S_{\sigma_{m}}\circ H_{\varepsilon_{m}}\circ S_{\sigma_{m-1}}\circ H_{\varepsilon_{m-1}}...\circ S_{\sigma_{2}}\circ H_{\varepsilon_{2}}\circ S_{\sigma_{1}}\circ H_{\varepsilon_{1}}\left(\phi\right)\label{eq: Model composition def}\end{equation}

\subsection{\label{sub: AsynchronousClusterStates}Existence and Stability of
Asynchronous Periodic States in Meta-Oscillator Networks}
\begin{defn}
An \emph{asynchronous periodic state} of a network of $N$ pulse-coupled
oscillators is a state $\Phi\in\mathcal{S}$ which is invariant under
the return map, i.e. $\mathbf{M}\left(\Phi\right)=\Phi$, and with
avalanche sizes $a_{s}=1$, $s\in\left\{ 1,2,\dots,N\right\} $, i.e.
each oscillator generates a pulse separately and does not receive
supra-threshold input. 
\end{defn}
Initially assume that all clusters stay forward invariant, i.e. do
not decay in to smaller sub-clusters during the network dynamics (cf.
sec. \ref{sub:Strategy}) and thus consider networks of meta-oscillators
with effective coupling matrix \eqref{eq: effective coupling}. A
periodic cluster state in the original model thus becomes an periodic
asynchronous state in the reduced effective meta-network. In the following
we derive conditions for the existence of the asynchronous state and
its stability in a meta-network. 
\begin{lem}
Consider a network \eqref{eq: Model interaction}-\eqref{eq: Model phase partial reset}
of $N$ oscillators with pulse coupling matrix \eqref{eq: effective coupling}
and neuronal partial reset. Let $\Sigma=\left(\sigma_{1},\dots,\sigma_{N}\right)\in\mathbb{R}^{N}$
and define $\mathbf{L}:\mathbb{R}^{N}\times S^{1}\rightarrow\mathbb{R}^{N}$
by \[
\mathbf{L}_{i}(\Sigma,\phi):=\bigodot_{s=i+1}^{N+i-1}\left(S_{\sigma_{s}}\circ H_{\varepsilon_{s}}\right)\circ S_{\sigma_{i}}\circ J_{\varepsilon_{ii}}(\phi)\]
for $i\in\left\{ 1,2,\dots,N\right\} $. Then the asynchronous state
exists if and only if there is a solution $\Sigma^{*}\in\mathbb{R}^{N}$
to the equation \begin{equation}
\mathbf{L}(\Sigma,1)=\left(1,1,\dots,1\right)\label{eq: SigmaStar}\end{equation}
that satisfies $\sigma_{r}^{*}>0$ for all $r\in\left\{ 1,2,\dots,N\right\} $.\end{lem}
\begin{proof}
Assume there is a solution $\Sigma^{*}$, $\sigma_{i}^{*}>0$. Set
\[
\phi_{1}^{*}=1,\quad\quad\phi_{i}^{*}=\left(H_{\varepsilon_{1}}^{-1}\circ S_{\sigma_{1}^{*}}^{-1}\right)\circ\left(H_{\varepsilon_{2}}^{-1}\circ S_{\sigma_{2}^{*}}^{-1}\right)\circ\dots\circ\left(H_{\varepsilon_{i-1}}^{-1}\circ S_{\sigma_{i-1}^{*}}^{-1}\right)(1)\]
for $i\in\left\{ 2\dots,N\right\} $. Using that $\Sigma^{*}$ is
a solution to \eqref{eq: SigmaStar} we have $\bigodot_{r=1}^{N-1}\left(S_{\sigma_{r}^{*}}\circ H_{\varepsilon_{r}}\right)\circ S_{\sigma_{N}^{*}}\circ J_{\varepsilon_{NN}}(1)=1$
and $\phi_{N}^{*}=S_{\sigma_{N}^{*}}\circ J_{\varepsilon_{NN}}(1)>0$
since $\sigma_{N}^{*}>0$. Further using $\varepsilon_{i}>0$ and
$\sigma_{r}^{*}>0$ the phases are ordered according to $\phi_{1}^{*}=1>\phi_{2}^{*}>\dots>\phi_{N}^{*}>0$
and $\Phi^{*}=\left(\phi_{1}^{*},\phi_{2}^{*},\dots,\phi_{N}^{*}\right)\in\mathcal{S}$. 

Starting from the state $\Phi^{*}$ the first pulse of oscillator
$i=1$ results in potentials $u_{1}^{(1)}=R\left(\varepsilon_{11}\right)$,
$u_{i}^{(1)}=U\left(\phi_{i}^{*}\right)+\varepsilon_{1}$, $i\in\left\{ 2,3,\dots,N\right\} $.
Since $R\left(\varepsilon_{ii}\right)\le\varepsilon_{ii}\le\varepsilon_{i}<\varepsilon_{i}+U(\phi)$
for all $\phi>0$ and $H_{\varepsilon_{i}}(\phi)<H_{\varepsilon_{i}}(\psi)$
for $\phi<\psi$ we have $u_{1}^{(1)}<u_{N}^{(1)}<u_{N-1}^{(1)}<\dots<u_{2}^{(1)}$.
Further \[
u_{2}^{(1)}=U\left(\phi_{2}^{*}\right)+\varepsilon_{1}=U\left(\left(H_{\varepsilon_{1}}^{-1}\circ S_{\sigma_{1}^{*}}^{-1}\right)(1)\right)+\varepsilon_{1}=U\left(1-\sigma_{1}^{*}\right)<1\]
as $\sigma_{1}^{*}>0$. Thus oscillator $i=1$ fires without triggering
any further oscillators yielding an avalanche set $\Theta_{1}=$ $\left\{ 1\right\} $.
In addition the oscillators have to be shifted by $\sigma_{1}^{*}$
to obtain $\phi_{2}=1$. Thus the first pair in the firing sequence
is $\left(\left\{ 1\right\} ,\sigma_{1}^{*}\right)$. Applying the
same arguments to the new phases $\Phi^{*(1)}=\mathbf{K}\left(\Phi^{*}\right)$
yields $\left(\left\{ 2\right\} ,\sigma_{2}^{*}\right)$ for the second
pair. Repeating these steps $N$ times one obtains a firing sequence
\[
\mathcal{F}\left(\Phi^{*}\right)=\left\{ \left(\left\{ r\right\} ,\sigma_{r}^{*}\right)\right\} _{r=1}^{N}\]
Thus\begin{eqnarray*}
\mathbf{M}_{i}\left(\Phi^{*}\right) & = & \bigodot_{r=i+1}^{N}\left(S_{\sigma_{r}^{*}}\circ H_{\varepsilon_{r}}\right)\circ S_{\sigma_{i}^{*}}J_{\varepsilon_{ii}}\circ\bigodot_{r=1}^{i-1}\left(S_{\sigma_{r}^{*}}\circ H_{\varepsilon_{r}}\right)\left(\phi_{i}^{*}\right)\\
 & = & \bigodot_{r=i+1}^{N}\left(S_{\sigma_{r}^{*}}\circ H_{\varepsilon_{r}}\right)\circ S_{\sigma_{i}^{*}}J_{\varepsilon_{ii}}(1)\\
 & = & \left(H_{\varepsilon_{1}}^{-1}\circ S_{\sigma_{1}^{*}}^{-1}\right)\circ\left(H_{\varepsilon_{2}}^{-1}\circ S_{\sigma_{2}^{*}}^{-1}\right)\circ\dots\circ\left(H_{\varepsilon_{i-1}}^{-1}\circ S_{\sigma_{i-1}^{*}}^{-1}\right)(1)\\
 & = & \phi_{i}^{*}\end{eqnarray*}
using \eqref{eq: SigmaStar} in the third row. Hence $\mathbf{M}\left(\Phi^{*}\right)=\Phi^{*}$
and the asynchronous state $\Phi^{*}$ is invariant under the return
map.

Conversely an periodic asynchronous state yields a solution to \eqref{eq: SigmaStar},
since by definition no oscillator receives supra-threshold input and
thus there are a phase shifts $\sigma_{i}>0$ after each pulse generation
of oscillators $i\in\left\{ 1,2,\dots,N\right\} $. Invariance of
the periodic asynchronous state then shows that in fact $\Sigma=\left(\sigma_{1},\sigma_{2},\dots,\sigma_{N}\right)$
is a solution to \eqref{eq: SigmaStar}. Hence there is no periodic
asynchronous state if the solution does not exist. If there is a solution
with $\sigma_{i}^{*}\le0$, let $s$ be the smallest index such that
$\sigma_{s}^{*}\le0$. Starting in the state $\Phi^{*}$ the first
firing of oscillator $i=s$ will cause oscillator $i=s+1$ to fire
in the same avalanche since its potential at this point is given by
\begin{eqnarray*}
u_{s+1}^{(1)} & = & U\left[\bigodot_{r=1}^{s-1}\left(S_{\sigma_{r}^{*}}\circ H_{\varepsilon_{r}}\right)\left(\phi_{s+1}^{*}\right)\right]+\varepsilon_{s}\\
 & = & U\left[H_{\varepsilon_{s}}^{-1}\circ S_{\sigma_{s}^{*}}^{-1}\left(1\right)\right]+\varepsilon_{s}=U\left(1-\sigma_{s}^{*}\right)\ge1\end{eqnarray*}
 i.e. $\left\{ r,r+1\right\} \subset\Theta_{r}$ and the system is
not in a periodic asynchronous state. \end{proof}
\begin{cor}
\label{lem: existence asynchronous state}In a network \eqref{eq: Model interaction}-\eqref{eq: Model phase partial reset}
of $N$ oscillators with homogeneous all-to-all coupling matrix \eqref{eq: All-to-All}
an asynchronous (splay) state exists.\begin{proof}
Let\[
L(\sigma):=\bigodot_{s=1}^{N-1}\left(S_{\sigma}\circ H_{\varepsilon}\right)\circ S_{\sigma}\circ J_{0}(1)\]
Now since $L(0)=U^{-1}\left(\varepsilon(N-1)\right)<1$ and $\frac{\partial}{\partial\sigma}L(\sigma)\ge1$
the intermediate value theorem ensures the existence of a $\sigma^{*}>0$
satisfying $L\left(\sigma^{*}\right)=1$. $\Sigma^{*}=\left(\sigma^{*},\dots,\sigma^{*}\right)$
is a solution to \eqref{eq: SigmaStar}. This result is independent
of the partial reset $R$ as $\varepsilon_{ii}=0$ and no oscillator
receives supra-threshold input in the asynchronous state.
\end{proof}
\emph{In Fig. \ref{fig: Desync prob} we observe no cluster states
involving avalanches of size 43 to 49. This is precisely because \ref{eq: SigmaStar}
has no solutions when setting $\varepsilon_{i}=a_{i}\varepsilon$,
$\varepsilon_{ii}=(a_{i}-1)\varepsilon$ for $a_{1}\in\left\{ 43,44,\dots,49\right\} $
and any further $0<a_{i}\in\mathbb{N}$, $i\ge1$ and $m$ such that
$\sum_{s=1}^{m}a_{s}=50$. }

\emph{Note that lemma \ref{lem: existence asynchronous state} holds
for any rise function $U$. If there are $q$ different positive solutions
to \eqref{eq: SigmaStar} there coexist $q$ different periodic asynchronous
states. A convex $U$ ensures that the solution is unique because
$\mathbf{L}\left(\Sigma,1\right)$ then becomes invertible for all
$\Sigma\in\mathbb{R}^{N}$. Another consequence of convexity is that}
\emph{given the existence of an asynchronous state in a meta-oscillator
network it is linearly stable as the following theorem shows:}
\end{cor}
\begin{thm}
\label{thm: Analyse cluster states}Consider a network \eqref{eq: Model interaction}-\eqref{eq: Model phase partial reset}
of $N$ oscillators with pulse coupling matrix \eqref{eq: effective coupling}
and neuronal partial reset. If a periodic asynchronous state exist
it is linearly stable.\begin{proof}
Existence of the asynchronous state $\left(\Phi^{*},\mathrm{id}\right)\in\mathcal{S}^{p}$
with $\Phi^{*}=\left(\phi_{2}^{*},\dots,\phi_{N}^{*}\right)$ implies
invariance under the return map $\mathbf{M}^{p}$, \begin{equation}
\mathbf{M}^{p}\left(\Phi^{*},\mathrm{id}\right)=\left(\Phi^{*},\mathrm{id}\right)\label{eq: Analyse asynchron return map}\end{equation}
 For the intermediate states we set\[
\left(\Phi^{(s)},\mathrm{\pi^{(s)}}\right):=\left(\mathbf{K}^{p}\right)^{s}\left(\Phi^{*},\mathrm{id}\right)\quad s\in\left\{ 0,1,2,\dots,N\right\} \]
 If oscillator $i$ generates a pulse all oscillators $j\ne i$ receive
the same input $\varepsilon_{i}$ and oscillator $i$ receives an
input $\varepsilon_{ii}\le\varepsilon_{i}$. Hence, using $R(\zeta)\le\zeta$,
we find that the oscillators do not change their firing order and
$\pi^{(s)}$ is a cyclic permutation to the left $\pi^{(s)}(i)=i-s$.

We show that the asynchronous state is linearly stable: Adding a perturbation
$\Delta^{(0)}=\left(\delta_{1}^{(0)},...,\delta_{N-1}^{(0)}\right)$
to the asynchronous state such that initially the phases are given
by\[
\Psi^{(0)}:=\left(\phi_{1}^{(0)},...,\phi_{N-1}^{(0)}\right)=\Phi^{*}+\Delta^{(0)}\]
 We take the perturbation to be sufficiently small such that the oscillators
still fire asynchronously, i.e. the avalanches are of size $a_{s}=1$
and the order of the events is preserved. In the following, terms
of order $\mathcal{O}\left(\left(\Delta^{(0)}\right)^{2}\right)$
are neglected which we indicate by a dot above the equality sign ($\doteq$).
After $s$ firing events the phases are\[
\Psi^{(s)}=\mathbf{K}_{\Phi}^{p}\left(\Psi^{(s-1)},\pi^{(s-1)}\right)\doteq\mathbf{K}_{\Phi}^{p}\left(\Phi^{(s-1)},\pi^{(r-1)}\right)+\Delta^{(s)}=\Phi^{(s)}+\Delta^{(s)}\]
 where \[
\Delta^{(s)}=\mathbf{A}^{(s)}\Delta^{(s-1)}\]
 is the phase perturbation before the next firing and $\mathbf{A}^{(s)}$
is the Jacobian matrix of $\mathbf{K}_{\Phi}^{p}$ at $\left(\Phi^{(s-1)},\pi^{(s-1)}\right)$\begin{equation}
A_{ij}^{(s)}=\frac{d\mathbf{K}_{i}^{p}}{d\phi_{j}}\left(\Phi^{(s-1)},\pi^{(s-1)}\right)\,.\label{eq: Analyse asynchron delta1 eq A delta0}\end{equation}
 Setting $\sigma=1-H\left(\psi_{2},\varepsilon_{\pi(1)}\right)$ the
phase part of the firing-map for $N\ge3$ is\begin{equation}
\boldsymbol{K}_{\Phi}^{p}\left(\Psi,\pi\right)=\left(\begin{array}{c}
H\left(\psi_{3},\varepsilon_{\pi(1)}\right)+\sigma\\
H\left(\psi_{4},\varepsilon_{\pi(1)}\right)+\sigma\\
\dots\\
H\left(\psi_{N},\varepsilon_{\pi(1)}\right)+\sigma\\
J\left(1,\varepsilon_{\pi(1)\pi(1)}\right)+\sigma\end{array}\right)^{T}\label{eq: Analyse firing map ansynchron}\end{equation}
 Inserting \eqref{eq: Analyse firing map ansynchron} into \eqref{eq: Analyse asynchron delta1 eq A delta0}
gives\begin{equation}
\mathbf{A}^{(s)}=\left(\begin{array}{ccccc}
-a_{2}^{(s)} & a_{3}^{(s)} & 0 & \dots & 0\\
-a_{2}^{(s)} & 0 & a_{4}^{(s)} & \ddots & \vdots\\
\vdots & \vdots & \ddots & \ddots & 0\\
-a_{2}^{(s)} & 0 & \dots & 0 & a_{N}^{(s)}\\
-a_{2}^{(s)} & 0 & \dots & 0 & 0\end{array}\right)\label{eq: Stability Matrix}\end{equation}
 with \begin{equation}
a_{i}^{(s)}=\frac{d}{d\phi}H_{\varepsilon_{s}}\left(\phi_{i}^{(s-1)}\right)=\frac{U'\left(\phi_{i}^{(s-1)}\right)}{U'\left(H_{\varepsilon_{s}}\left(\phi_{i}^{(s-1)}\right)\right)}\label{eq: aij}\end{equation}
 Since $\varepsilon_{j}>0$ it follows that $H_{\varepsilon_{j}}\left(\phi\right)=U^{-1}\left(U(\phi)+\varepsilon_{j}\right)>\phi$.
Thus $a_{i}^{(s)}<1$ since $U$ is convex. Also \textbf{$U'>0$}
and hence \begin{equation}
0<a_{i}^{(s)}<1\label{eq: Analyse asynchron 0<a<1}\end{equation}
 Now the Eneström-Kakeya theorem (cf. appendix \ref{sec: Appendix EnestromKakeya}
and \cite{Enestroem:1893,Kakeya:1912,AndersonVarga:1972,Horn:1996})
applied to the matrix $\mathbf{A}^{(s)}$ shows that with these properties
the spectral radius $\rho\left(\mathbf{A}^{(s)}\right)$ of $\mathbf{A}^{(s)}$
satisfies \[
\rho\left(\mathbf{A}^{(s)}\right)\le r^{(s)}=\max_{i\in\left\{ 1,...,N-1\right\} }a_{i}^{(s)}<1\]
Thus \begin{equation}
\left\Vert \Delta^{(nN)}\right\Vert =\left\Vert \left(\prod_{r=1}^{N}\mathbf{A}^{(s)}\right)^{n}\Delta^{(0)}\right\Vert \le\prod_{r=1}^{N}\rho\left(\mathbf{A}^{(s)}\right)^{n}\left\Vert \Delta^{(0)}\right\Vert \rightarrow0\quad\quad\textrm{as }n\rightarrow\infty\label{eq: Analyse asynchron delta-> 0}\end{equation}
and the asynchronous state is linearly stable. For $N=2$, $\rho\left(\mathbf{A}^{(s)}\right)=a_{2}<1$.
\end{proof}
\end{thm}
This result is illustrated in Fig. \ref{cap: Analyse Stability-of-the-Asynchronous-State}:
Due to the convexity of the rise function oscillators perturbed to
larger (smaller) phases compared to the asynchronous state are less
(more) advanced by input pulses pulling the perturbed phases back
to the invariant asynchronous dynamics.

\begin{figure}[!t]
\begin{centering}
\includegraphics[scale=0.8]{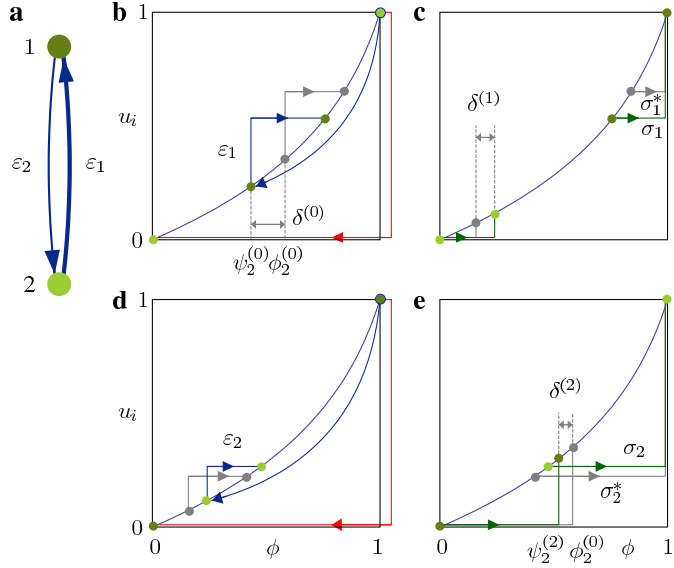} 
\par\end{centering}

\caption{\label{cap: Analyse Stability-of-the-Asynchronous-State}Stability
of the asynchronous state. \textbf{(a)} Graph a network of $N=2$
oscillators with connectivity $\varepsilon_{ij}=\left(1-\delta_{ij}\right)\varepsilon_{j}$
with $0<\varepsilon_{2}<\varepsilon_{1}$. \textbf{(b)} Firing of
oscillator $i=1$. For the oscillator $i=2$ with initial phase $\psi_{2}^{(0)}=\phi_{2}^{(0)}+\delta^{(0)}$
smaller than in the invariant asynchronous state $\phi_{2}^{(0)}$
(gray) the input advances the phase $\psi_{2}^{(0)}$ more in comparison
with the advance of $\phi_{2}^{(0)}$ in the asynchronous state due
to the convexity of the rise function. \textbf{(c)} After the interaction
a subsequent shift completes the firing map $\mathbf{K}$. In total
the derivation from the asynchronous state $\delta^{(1)}$ has become
smaller. \textbf{(d)} Firing of oscillator $i=2$. Phases which are
perturbed to larger values than the asynchronous state are less advanced
by inputs due to convexity of the rise function, i.e. $H_{\varepsilon_{1}}\left(\psi_{2}^{(0)}\right)-\psi_{2}^{(0)}>H_{\varepsilon_{1}}\left(\phi_{2}^{(0)}\right)-\phi_{2}^{(0)}$.
\textbf{(d)} In total the return map $\mathbf{M}$ decreases the phase
perturbations $\left|\delta^{(2)}\right|<\left|\delta^{(0)}\right|$.
These stabilizing dynamics of the asynchronous state due to the convexity
of the rise function generalizes to larger networks as proven in theorem
\ref{thm: Analyse cluster states}.}

\end{figure}

Combining corollary \ref{lem: existence asynchronous state} and theorem
\ref{thm: Analyse cluster states} we obtain: 
\begin{cor}
\label{cor: Analyse Asynchron stable}In a network \eqref{eq: Model interaction}-\eqref{eq: Model phase partial reset}
of $N$ oscillators with homogeneous all-to-all coupling matrix \eqref{eq: All-to-All},
neuronal partial reset $R$ and convex rise function $U$ the periodic
asynchronous (splay) state exists and is linearly stable. 
\end{cor}

\subsection{\label{sub:Stability-against-Cluster}Impact of Partial Reset on
Intra-Cluster Stability}

In the state of synchronous firing all units in an all-to-all coupled
network receive a supra-threshold input pulse of strength $\left(N-1\right)\varepsilon$
suggesting a rather strong influence of the partial reset $R$ onto
the network dynamics. Indeed, as shown in Fig. \ref{fig: Desync prob}
for the partial reset $R_{c}$ one observes a sequential destabilization
of clusters when increasing the reset strength $c$. In this subsection
we study this behavior analytically and explain the observed transition.
The strategy is to focus on a single cluster and derive general conditions
which ensure the stability of this cluster under the return map. As
the return map depends on the firing sequence $\mathcal{F}$ only
we give lower bounds below which cluster sizes are ensured to be stable
and upper bounds above which cluster sizes are unstable. However,
we find that for a special class of rise functions a full analytical
treatment is possible.
\begin{defn}
A firing sequence $\mathcal{F}$ is \emph{admissible} if there is
a state $\Phi\in\mathcal{S}$ which has firing sequence $\mathcal{F}=\mathcal{F}(\Phi)$.
It is further called \emph{trigger invariant }if for the oscillators
$i\in\Theta_{1}^{(0)}=\left\{ j\in\left\{ 1,2,\dots,N\right\} \,|\,\phi_{j}=1\right\} $
triggering the first avalanche of the state $\Phi=\left(\phi_{1},\dots,\phi_{N}\right)$
(cf. \eqref{eq: Model Avalanche triggering set}) the return map satisfies
$\mathbf{M}_{i}(\Phi)=1$. Thus for a trigger invariant firing sequence
$\mathcal{F}$ with $m$ intermediate avalanches $\Theta_{1}^{(0)}\subset\Theta_{m+1}^{(0)}$.
The set of all  trigger invariant firing sequences is denoted by
$\mathcal{T}$. The set of $\mathcal{F}\in\mathcal{T}$ with initial
avalanche size $a_{1}=\left|\Theta_{1}\right|$ is denoted by $\mathcal{T}_{a_{1}}$.
\end{defn}
Let us focus on a single avalanche of size $a_{1}$ in the network
dynamics. To ensure that all units in this avalanche fire together
again after the return map is applied all units in this avalanche
which were triggered to fire by $a\in\left\{ 1,2,\dots,a_{1}-1\right\} $
preceding spikes i.e. with phases in \[
I_{a}^{T}=\left[U^{-1}\left(1-a\varepsilon\right),1\right]\]
have to be triggered again after applying the return map. Given a
firing sequence $\mathcal{F}=\left\{ \left(\varepsilon_{r},\sigma_{r}\right)\right\} _{r=1}^{m}$
the return map for oscillators $i\in\Theta_{1}$ in the first avalanche
is given by \[
M_{\mathcal{F}}(\phi):=\bigodot_{r=2}^{m}\left(S_{\sigma_{r}}\circ H_{\varepsilon_{r}}\right)\circ S_{\sigma_{1}}\circ J_{\varepsilon_{1}}\left(\phi\right)\]
Hence the conditions\begin{equation}
\mathcal{M}_{\mathcal{F}}\left(I_{a}^{T}\right)\subset I_{a}^{T}\label{eq: conditions}\end{equation}
for all $a\in\left\{ 1,\dots,a_{1}-1\right\} $ and all admissible
firing sequences $\mathcal{F}\in\mathcal{T}_{a_{1}}$ ensure a cluster
of size $a_{1}$ to not split up under return.By finding the most
synchronizing and most desynchronizing firing sequences (i.e. {}``best''
and {}``worst case'' firing sequences) $\mathcal{F}\in\mathcal{T}_{a_{1}}$
these conditions yield upper and lower bounds for stability of a cluster
of size $a_{1}$ under the return map.
\begin{lem}
Consider a network \eqref{eq: Model interaction}-\eqref{eq: Model phase partial reset}
of $N$ oscillators with homogeneous all-to-all coupling matrix \eqref{eq: All-to-All}.

Set \[
w_{a}^{a_{1}}=\inf_{\mathcal{F}\in\mathcal{T}_{a_{1}}}M_{\mathcal{F}}\left(U^{-1}\left(1-a\varepsilon\right)\right)\]

and \[
b_{a}^{a_{1}}=\sup_{\mathcal{F}\in\mathcal{T}_{a_{1}}}M_{\mathcal{F}}\left(U^{-1}\left(1-a\varepsilon\right)\right)\]

Then the conditions \begin{equation}
w_{a}^{a_{1}}\ge U^{-1}\left(1-a\varepsilon\right)\label{eq: worst case cond}\end{equation}

for $a\in\left\{ 1,2,\dots,a_{1}-1\right\} $ are sufficient and \begin{equation}
b_{a}^{a_{1}}\ge U^{-1}\left(1-a\varepsilon\right)\label{eq: best case cond}\end{equation}

are necessary for a cluster of size $a_{1}$ to be invariant under
return.\end{lem}
\begin{proof}
$\frac{\partial}{\partial\phi}M_{\mathcal{F}}(\phi)>0$ and thus conditions
\eqref{eq: conditions} are equivalent to $\mathcal{M}_{\mathcal{F}}\left(U^{-1}\left(1-a\varepsilon\right)\right)\ge U^{-1}\left(1-a\varepsilon\right)$
for $a\in\left\{ 1,2,\dots,a_{1}-1\right\} $ and all admissible $\mathcal{F}\in\mathcal{T}_{a_{1}}$.
\end{proof}
Finding the $w_{a}^{a_{1}}$ and $b_{a}^{a_{1}}$ for general $U$
and $R$ can be done numerically using optimization techniques. However,
there are two classes of rise functions which allow further analytical
investigation. Most of the commonly used rise functions, as e.g. the
rise function of the LIF neuron or the conductance based LIF neuron
fall into one of these classes (cf. Appendix \ref{sec:Rise-Functions}).

Two oscillators initially at phases $\phi$ and $\phi+\Delta\phi$
receiving a pulse of strength $\varepsilon$ will have a new phase
difference \begin{equation}
\Delta H\left(\phi,\Delta\phi,\varepsilon\right):=H_{\varepsilon}\left(\phi+\Delta\phi\right)-H_{\varepsilon}\left(\phi\right)\;.\label{eq: Delta H}\end{equation}
We denote the domain of $\Delta H$ as \[
\mathcal{D}:=\left\{ \left(\phi,\Delta\phi,\varepsilon\right)\,|\,0\le\varepsilon\le1,\,0\le\phi\le1,\,0\le\Delta\phi\le U^{-1}\left(1-\varepsilon\right)-\phi\right\} \,.\]

\begin{defn}
\label{def:icpd dcpd}A rise function $U$ is \emph{increasing the
change of phase differences (icpd)} iff\begin{equation}
\frac{\partial}{\partial\phi}\Delta H\left(\phi,\Delta\phi,\varepsilon\right)\ge0\quad\text{for all}\,\left(\phi,\Delta\phi,\varepsilon\right)\in\mathcal{D}\,.\label{eq: Analyse Delta H early contracting}\end{equation}
 Conversely, it is \emph{decreasing the change of phase differences
(dcpd)} iff\begin{equation}
\frac{\partial}{\partial\phi}\Delta H\left(\phi,\Delta\phi,\varepsilon\right)\le0\quad\text{for all}\,\left(\phi,\Delta\phi,\varepsilon\right)\in\mathcal{D}\:.\label{eq: Analyse Delta H late constracting}\end{equation}
As shown in appendix \ref{sub:Icpd-and-Dcpd} the icpd (dcpd) property
is related to the third derivative of $U$. 
\end{defn}
The following lemma allows to bound the change in phase differences
if the rise function is icpd or dcpd: 
\begin{lem}
\label{lem: Analyse e/l contracting commutation relations}Let $\varepsilon_{r}$,
$\sigma_{r}\ge0$, $r\in\left\{ 1,2,\dots,m\right\} $, $\varepsilon=\sum_{r=1}^{m}\varepsilon_{r}$,
$\sigma_{l}\ge0$. Choose a $\sigma_{u}\ge0$ such that\[
\bigodot_{r=1}^{m}\left(S_{\sigma_{r}}\circ H_{\varepsilon_{r}}\right)\left(\phi\right)\le H_{\varepsilon}\circ S_{\sigma_{u}}\left(\phi\right)\,.\]
and let $\psi\le\phi$. Then for an icpd rise function $U$\begin{equation}
S_{\sigma_{l}}\circ H_{\varepsilon}\left(\phi\right)-S_{\sigma_{l}}\circ H_{\varepsilon}\left(\psi\right)\,\le\,\bigodot_{r=1}^{m}\left(S_{\sigma_{r}}\circ H_{\varepsilon_{r}}\right)\left(\phi\right)-\bigodot_{r=1}^{m}\left(S_{\sigma_{r}}\circ H_{\varepsilon_{r}}\right)\left(\psi\right)\,\le\, H_{\varepsilon}\circ S_{\sigma_{u}}\left(\phi\right)-H_{\varepsilon}\circ S_{\sigma_{u}}\left(\psi\right)\label{eq: Analyse less idp}\end{equation}
 If $U$ is dcpd then \begin{equation}
S_{\sigma_{l}}\circ H_{\varepsilon}\left(\phi\right)-S_{\sigma_{l}}\circ H_{\varepsilon}\left(\psi\right)\ge\bigodot_{r=1}^{m}\left(S_{\sigma_{r}}\circ H_{\varepsilon_{r}}\right)\left(\phi\right)-\bigodot_{r=1}^{m}\left(S_{\sigma_{r}}\circ H_{\varepsilon_{r}}\right)\left(\psi\right)\ge H_{\varepsilon}\circ S_{\sigma_{u}}\left(\phi\right)-H_{\varepsilon}\circ S_{\sigma_{u}}\left(\psi\right)\label{eq: Analyse greater dpd}\end{equation}
 \begin{proof}
Consider icpd rise functions first: To show the first inequality of
eq. (\ref{eq: Analyse less idp}) we use induction on $m$. The statement
is clearly true for $m=1$. Assume it is true for $m\ge1$ then \begin{align*}
S_{\sigma_{l}}\circ H_{\varepsilon}\left(\phi\right)-S_{\sigma_{l}}\circ & H_{\varepsilon}\left(\psi\right)=H_{\varepsilon_{m+1}}\circ H_{\varepsilon-\varepsilon_{m+1}}\left(\phi\right)-H_{\varepsilon_{m+1}}\circ H_{\varepsilon-\varepsilon_{m+1}}\left(\psi\right)\\
= & \Delta H\left(H_{\varepsilon-\varepsilon_{m+1}}\left(\psi\right),H_{\varepsilon-\varepsilon_{m+1}}\left(\phi\right)-H_{\varepsilon-\varepsilon_{m+1}}\left(\psi\right),\varepsilon_{m+1}\right)\\
\le & \Delta H\left(H_{\varepsilon-\varepsilon_{m+1}}\left(\psi\right),\bigodot_{r=1}^{m}\left(S_{\sigma_{r}}\circ H_{\varepsilon_{r}}\right)\left(\phi\right)-\bigodot_{r=1}^{m}\left(S_{\sigma_{r}}\circ H_{\varepsilon_{r}}\right)\left(\psi\right),\varepsilon_{m+1}\right)\\
\le & \Delta H\left(\bigodot_{r=1}^{m}\left(S_{\sigma_{r}}\circ H_{\varepsilon_{r}}\right)\left(\psi\right),\bigodot_{r=1}^{m}\left(S_{\sigma_{r}}\circ H_{\varepsilon_{r}}\right)\left(\phi\right)-\bigodot_{r=1}^{m}\left(S_{\sigma_{r}}\circ H_{\varepsilon_{r}}\right)\left(\psi\right),\varepsilon_{m+1}\right)\\
= & \bigodot_{r=1}^{m+1}\left(S_{\sigma_{r}}\circ H_{\varepsilon_{r}}\right)\left(\phi\right)-\bigodot_{r=1}^{m+1}\left(S_{\sigma_{r}}\circ H_{\varepsilon_{r}}\right)\left(\psi\right)\end{align*}
where we used the induction hypothesis and $\frac{\partial}{\partial\Delta\phi}\Delta H>0$
(cf. \eqref{eq: Delta H}) in the third, and in the fourth again the
icpd property and the fact that $H_{\varepsilon-\varepsilon_{m+1}}\left(\psi\right)\le\bigodot_{r=1}^{m}\left(S_{\sigma_{r}}\circ H_{\varepsilon_{r}}\right)$
if $\sum_{r=1}^{m+1}\varepsilon_{r}=\varepsilon$, $\sigma_{i}\ge0$.
Substituting $\le$ with $\ge$ we obtain the result for dcpd rise
functions.

For the second inequality we also use induction over $m$. The statement
is trivially true for $m=1$. Let it be true for $m\ge1$ and let
$\sigma_{u}\ge0$ such that $\bigodot_{r=1}^{m+1}\left(S_{\sigma_{r}}\circ H_{\varepsilon_{r}}\right)\left(\phi\right)\le H_{\varepsilon}\circ S_{\sigma_{u}}\left(\phi\right)$.
Then \begin{align*}
H_{\varepsilon}\circ S_{\sigma_{u}}\left(\phi\right)-H_{\varepsilon} & \circ S_{\sigma_{u}}\left(\psi\right)=H_{\varepsilon_{m+1}}\circ H_{\varepsilon-\varepsilon_{m+1}}\circ S_{\sigma_{u}}\left(\phi\right)-H_{\varepsilon_{m+1}}\circ H_{\varepsilon-\varepsilon_{m+1}}\circ S_{\sigma_{u}}\left(\psi\right)\\
 & =\Delta H\left(H_{\varepsilon-\varepsilon_{m+1}}\circ S_{\sigma_{u}}\left(\psi\right),H_{\varepsilon-\varepsilon_{m+1}}\circ S_{\sigma_{u}}\left(\phi\right)-H_{\varepsilon-\varepsilon_{m+1}}\circ S_{\sigma_{u}}\left(\psi\right),\varepsilon_{m+1}\right)\\
 & \ge\Delta H\left(H_{\varepsilon-\varepsilon_{m+1}}\circ S_{\sigma_{u}}\left(\psi\right),\bigodot_{r=1}^{m}\left(S_{\sigma_{r}}\circ H_{\varepsilon_{r}}\right)\left(\phi\right)-\bigodot_{r=1}^{m}\left(S_{\sigma_{r}}\circ H_{\varepsilon_{r}}\right)\left(\psi\right),\varepsilon_{m+1}\right)\\
 & \ge\Delta H\left(\bigodot_{r=1}^{m}\left(S_{\sigma_{r}}\circ H_{\varepsilon_{r}}\right)\left(\psi\right),\bigodot_{r=1}^{m}\left(S_{\sigma_{r}}\circ H_{\varepsilon_{r}}\right)\left(\phi\right)-\bigodot_{r=1}^{m}\left(S_{\sigma_{r}}\circ H_{\varepsilon_{r}}\right)\left(\psi\right),\varepsilon_{m+1}\right)\\
 & =\bigodot_{r=1}^{m+1}\left(S_{\sigma_{r}}\circ H_{\varepsilon_{r}}\right)\left(\phi\right)-\bigodot_{r=1}^{m+1}\left(S_{\sigma_{r}}\circ H_{\varepsilon_{r}}\right)\left(\psi\right)\end{align*}
 where in the third row we used the implication\begin{equation}
\bigodot_{s=1}^{m+1}\left(S_{\sigma_{s}}\circ H_{\varepsilon_{s}}\right)\left(\phi\right)\le H_{\varepsilon}\circ S_{\sigma_{u}}\left(\phi\right)\,\Rightarrow\,\bigodot_{s=1}^{m}\left(S_{\sigma_{s}}\circ H_{\varepsilon_{s}}\right)\left(\phi\right)\le H_{\varepsilon-\varepsilon_{m+1}}\circ S_{\sigma_{u}}\left(\phi\right)\label{eq: Analyse e-contracting condition greater}\end{equation}
to apply the induction hypothesis. In the fourth row we again used
$\frac{\partial}{\partial\Delta\phi}\Delta H>0$, eq. \eqref{eq: Analyse e-contracting condition greater}
and the icpd property. Substituting $\ge$ with $\le$ we obtain the
result for dcpd rise functions. 
\end{proof}
\end{lem}
This result is illustrated in Fig. \ref{cap: Analyse idp Rise-functions}
(a-c) for icpd rise functions. 

\begin{figure}[!t]
\begin{centering}
\includegraphics[scale=0.8]{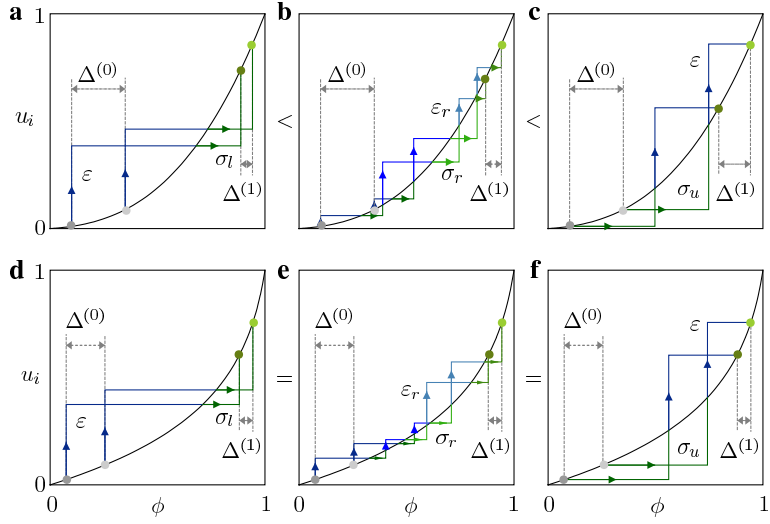} 
\par\end{centering}

\caption{\label{cap: Analyse idp Rise-functions}Rise functions with increasing
change (icpd) and no change of phase differences. \textbf{(a-c)} Icpd
rise function. An initial phase difference $\Delta^{(0)}$ changes
to $\Delta^{(1)}$ after applying a combination of interaction maps
$H_{\varepsilon_{r}}$ (blue) of total strength $\varepsilon=\sum_{r=1}^{m}\varepsilon_{r}$
and shifts $S_{\sigma_{r}}$ (green) such that the final maximal phase
values are identical. \textbf{(a) }For icpd rise functions the difference
$\Delta^{(1)}$ is the smallest when the interaction is applied in
total before the shifts, i.e. $H_{\varepsilon}\circ S_{\sigma_{l}}$
and \textbf{(c)} largest when applied after the shifts $S_{\sigma_{u}}\circ H_{\varepsilon}$.
\textbf{(b)} All other maps $\bigodot_{s=1}^{m}\left(H_{\varepsilon_{s}}\circ S_{\sigma_{s}}\right)$
produce phase differences which lie in between these extremal values
(cf. lemma \ref{lem: Analyse e/l contracting commutation relations}).
\textbf{(d-f)} The rise function $U_{b}$ is icpd and dcpd, i.e. the
phase difference $\Delta\phi^{(1)}$ is independent of the order in
which the interactions and shifts are applied.}

\end{figure}

The next theorem allows to determine bounds on the network parameters
in order to ensure invariance or decay of avalanches under return: 
\begin{thm}
\label{thm: Analyse Avalanche stability icpd dcpd}Consider a homogenous
excitatory all-to-all network of $N$ pulse-coupled oscillators evolving
according to \eqref{eq: Model interaction}-\eqref{eq: Model phase partial reset}
with neuronal partial reset $R$. 

For icdp rise functions $U$ the conditions \begin{equation}
U^{-1}\left(R\left((a_{1}-1)\varepsilon\right)\right)-U^{-1}\left(R\left((a_{1}-1)\varepsilon-a\varepsilon\right)\right)\le U^{-1}\left(1-(N-a_{1})\varepsilon\right)-U^{-1}\left(1-(N-a_{1})\varepsilon-a\varepsilon\right)\label{eq: Analyse icdp a-avlanche stable}\end{equation}
for all $a\in\left\{ 1,2,\dots,a_{1}-1\right\} $ are sufficient to
ensure the invariance of an $a_{1}$-avalanche under return. Necessary
conditions are 

\begin{equation}
U^{-1}\left(R\left((a_{1}-1)\varepsilon\right)+(N-a_{1})\varepsilon\right)-U^{-1}\left(R\left((a_{1}-1)\varepsilon-a\varepsilon\right)+(N-a_{1})\varepsilon\right)\le1-U^{-1}\left(1-a\varepsilon\right)\:.\label{eq: Analyse dcpd a-avlanche stable}\end{equation}

Likewise for dcpd rise functions $U$ sufficient conditions are \eqref{eq: Analyse dcpd a-avlanche stable}
and necessary conditions \eqref{eq: Analyse icdp a-avlanche stable}
for an $a_{1}$-avalanche to not split up under return. \end{thm}
\begin{proof}
Using lemma \ref{lem: Analyse e/l contracting commutation relations}
we find for an icpd rise function and $\mathcal{F}\in\mathcal{T}_{a_{1}}$

\begin{eqnarray*}
M_{\mathcal{F}}(1)-M_{\mathcal{F}}\left(U^{-1}\left(1-a\varepsilon\right)\right) & = & \bigodot_{r=2}^{m}\left(S_{\sigma_{r}}\circ H_{\varepsilon_{r}}\right)\left(S_{\sigma_{1}}\circ J_{\varepsilon_{1}}\left(1\right)\right)-\bigodot_{r=2}^{m}\left(S_{\sigma_{r}}\circ H_{\varepsilon_{r}}\right)\left(S_{\sigma_{1}}\circ J_{\varepsilon_{1}}\left(U^{-1}\left(1-a\varepsilon\right)\right)\right)\\
 & \le & H_{\left(N-a_{1}\right)\varepsilon}\circ S_{\sigma_{u}}\circ J_{\varepsilon_{1}}\left(1\right)-H_{\left(N-a_{1}\right)\varepsilon}\circ S_{\sigma_{u}}\circ J_{\varepsilon_{1}}\left(U^{-1}\left(1-a\varepsilon\right)\right)\\
 & = & 1-H_{\left(N-a_{1}\right)\varepsilon}\circ S_{\sigma}\circ J_{\varepsilon_{1}}\left(U^{-1}\left(1-a\varepsilon\right)\right)\end{eqnarray*}
with \begin{equation}
\sigma_{u}=U^{-1}\left(1-\left(N-a_{1}\right)\varepsilon\right)-U^{-1}\left(R\left(a_{1}-1\right)\varepsilon\right)\label{eq: sigma u}\end{equation}
and thus \[
w_{a}^{a_{1}}=H_{\left(N-a_{1}\right)\varepsilon}\circ S_{\sigma_{u}}\circ J_{\varepsilon_{1}}\left(U^{-1}\left(1-a\varepsilon\right)\right)\]
in \eqref{eq: worst case cond} yielding conditions \eqref{eq: Analyse icdp a-avlanche stable}.
Similarly we find \textbf{\[
b_{a}^{a_{1}}=1-H_{\left(N-a_{1}\right)\varepsilon}\circ J_{\varepsilon_{1}}\left(1\right)+H_{\left(N-a_{1}\right)\varepsilon}\circ J_{\varepsilon_{1}}\left(U^{-1}\left(1-a\varepsilon\right)\right)\]
}which yields the necessary conditions \eqref{eq: Analyse dcpd a-avlanche stable}.
For dcdp rise functions the expressions for $w_{a}^{a_{1}}$ and $b_{a}^{a_{1}}$
are interchanged. \end{proof}
\begin{prop}
\label{prop: Analyse Bounds Cluster instability}In a homogeneous
all-to-all coupled network of $N$ neural oscillators evolving according
to \eqref{eq: Model interaction}-\eqref{eq: Model phase partial reset}
with neuronal partial reset $R$ and dcdp rise function $U$ the condition
\begin{equation}
R'(\zeta)>\frac{U'\left(U^{-1}\left(R(\zeta)\right)\right)}{U'\left(U^{-1}\left(1-(N-1)\varepsilon+\zeta\right)\right)}\quad\text{for all}\:(a_{1}-2)\varepsilon\le\zeta\le(a_{1}-1)\varepsilon\label{eq: Analyse Cluster instability}\end{equation}
is sufficient to ensure non-invariance of an $a_{1}$-cluster under
return.\begin{proof}
Using lemma \ref{lem: Analyse e/l contracting commutation relations}
for a given firing sequence $\mathcal{F}=\left\{ \left(a_{s}\varepsilon,\sigma_{s}\right)\right\} _{s=1}^{m}$
we estimate\begin{equation}
1-H_{\left(N-a_{1}\right)\varepsilon}\circ S_{\sigma_{u}}\circ J_{\varepsilon_{1}}\left(1-\Delta\phi\right)\le1-M_{\mathcal{F}}\left(1-\Delta\phi\right)\le1-S_{\sigma_{l}}\circ H_{\left(N-a_{1}\right)\varepsilon}\circ J_{\varepsilon_{1}}\left(1-\Delta\phi\right)\label{eq: Analyse Delta_M bounds 2}\end{equation}
 with $\sigma_{u}$ as in \eqref{eq: sigma u} and $\sigma_{l}=1-H_{\left(N-a_{1}\right)\varepsilon}\left(J_{a_{1}\varepsilon}(1)\right)$.
In general a $a_{1}$-cluster is triggered by a single oscillator
and thus if\begin{equation}
1-H_{\left(N-a_{1}\right)\varepsilon}\circ S_{\sigma_{u}}\circ J_{a_{1}}\left(1-\Delta\phi\right)>\Delta\phi\label{eq: hhh}\end{equation}
for all $0<\Delta\phi\le1-U^{-1}\left(1-\varepsilon\right)$, we have
$1-M_{\mathcal{F}}\left(1-\Delta\phi\right)>\Delta\phi$ which implies
that after a finite number of iterations of the return map the firing
of the first oscillator does not trigger the avalanche any more. Setting
$\zeta=U\left(1-\Delta\phi\right)+(a_{1}-1)\varepsilon-1$ condition
\eqref{eq: hhh} is equivalent to \begin{equation}
U^{-1}\left(R\left((a_{1}-1)\varepsilon\right)\right)-U^{-1}\left(R\left(\zeta\right)\right)>U^{-1}\left(1-(N-a)\varepsilon\right)-U^{-1}\left(1-(N-1)\varepsilon+\zeta\right)\label{eq: hkk}\end{equation}
for all $(a_{1}-2)\varepsilon\le\zeta<(a_{1}-1)\varepsilon$. For
$\zeta=(a_{1}-1)\varepsilon$ both sides in \eqref{eq: hkk} are equal
and the condition \eqref{eq: Analyse Cluster instability} thus ensures
\eqref{eq: hkk} to hold for all $(a_{1}-2)\varepsilon\le\zeta<(a_{1}-1)\varepsilon$.
\end{proof}
\end{prop}
Note that for convex $U$ and neuronal $R$ the right hand side of
inequality \eqref{eq: Analyse Cluster instability} is smaller than
one and thus a sufficient condition for $a_{1}$-avalanches to split
up after a finite number of applications of the return map is \[
R'(\zeta)\ge1\quad\text{for all }\,\zeta\in\left[(a_{1}-2)\varepsilon,(a_{1}-1)\varepsilon\right]\]
In particular for a partial reset $R_{c=1}(\zeta)=\zeta$ all avalanches
become unstable under the return map and thus by corollary \ref{cor: Analyse Asynchron stable}
only the asynchronous state remains stable for all convex dcpd $U$
and $c=1$. 

We used theorem \ref{thm: Analyse cluster states} and proposition
\ref{prop: Analyse Bounds Cluster instability} to determine for a
convex LIF rise function $U_{\mathrm{LIF}}^{\mathrm{CB}}$ (cf. \ref{eq: Model U CB LIF}
eq. \eqref{eq: Model U CB QIF}) and linear partial reset $R_{c}$
the regime where avalanches of different sizes become unstable under
return. The most strict condition in \eqref{eq: Analyse icdp a-avlanche stable}
is for $a=1$ which yields an implicit equation for the lower bounds
on the critical $c$ values below which the invariance of $a_{1}$-avalanches
is ensured. The upper bound is obtained by \eqref{eq: Analyse dcpd a-avlanche stable}
using $a=1$ and is very close to the bound given in proposition \ref{prop: Analyse Bounds Cluster instability}.
The bounds are plotted in fig. \ref{cap: Analyse Bounds-for-clustersizes}
and are in good agreement with the numerical data. 

\begin{figure}[!b]
\begin{centering}
\includegraphics{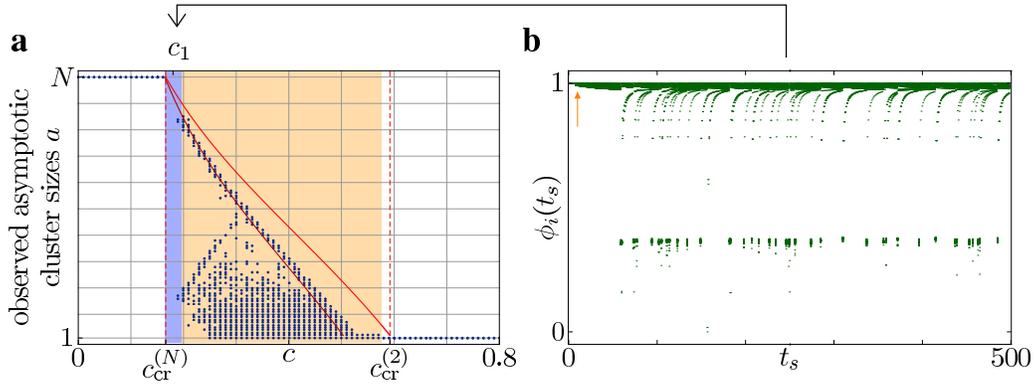}
\par\end{centering}

\caption{\label{cap: Analyse Bounds-for-clustersizes}Sequential desynchronization
in a network ($N=100$) with icpd rise function $U_{\mathrm{LIF}}^{\mathrm{CB}}$\textbf{
}($E_{\mathrm{eq}}=1.1$, $E_{\mathrm{syn}}=3$) and linear partial
reset $R_{c}$. \textbf{(a)} Observed cluster sizes of periodic states
after a time $t=10000$. For each $c$ value $100$ simulations were
started in the synchronous state with a small perturbation added.
The upper line shows the bounds on $a$ obtained from \eqref{eq: Analyse dcpd a-avlanche stable}
in theorem \ref{thm: Analyse Avalanche stability icpd dcpd} above
which $a$-clusters are unstable. The lower line is the bound obtained
via \eqref{eq: Analyse icdp a-avlanche stable} below which $a$-clusters
are ensured to be stable. The shaded area marks the transition region
where states other than the synchronous an asynchronous state are
observed. In the blue region we find no periodic asymptotic dynamics.
The dashed lines show the theoretical bounds for the transition region.
\textbf{(b)} Aperiodic dynamics for $c_{1}=0.18$.}

\end{figure}

Near the lower transition point $c_{\mathrm{crit}}^{(N)}$ the system
shows aperiodic behavior when starting close to the synchronous state.
A possible explanation for this dynamics is the competition of two
counteracting mechanisms: (i) Large avalanches become unstable under
return and thus tend to desynchronize the phases which results in
a split of the avalanche into smaller stable avalanches. (ii) The
solution to equation \eqref{eq: SigmaStar} for these asynchronously
firing smaller clusters involves $\sigma_{r}^{*}\le0$, i.e. the smaller
avalanches tend to absorb each other and resynchronize the system
yielding again larger unstable avalanches. Note that here irregular
dynamics arise via a mechanism different form as for example network
heterogeneity \cite{Denker:2004} or using excitatory and inhibitory
interactions \cite{vanVreeswijk:1996a}.

\subsection{\label{sub:Solvable Example}Extensive Sequence of Desynchronizing
Bifurcations -- A Solvable Example}

Figure \ref{cap: Analyse idp Rise-functions} (d-f) illustrates that
the rise function $U_{b}$ is both icpd and dcpd. In fact, \[
\Delta H_{b}\left(\phi,\Delta\phi,\varepsilon\right):=H_{b}\left(\phi+\Delta\phi,\varepsilon\right)-H_{b}\left(\phi,\varepsilon\right)=\Delta\phi e^{b\varepsilon}\]
 is independent of $\phi$ and hence \[
\frac{\partial}{\partial\phi}\Delta H_{b}\left(\phi,\Delta\phi,\varepsilon\right)=0\:.\]
Thus for $U_{b}$ equality holds in \eqref{eq: Analyse less idp}
and the ''best''- and ''worst-case'' return maps become identical.
This property allows to obtain exact analytical results.
\begin{prop}
\label{thm: Analyse Desynchronization U_b R_c}Consider a homogenous
excitatory all-to-all network of $N$ pulse-coupled oscillators evolving
according to \eqref{eq: Model interaction}-\eqref{eq: Model phase partial reset}
with convex rise function $U_{b}$ ($b<0$) and neuronal partial reset
$R_{c}$. 

Then for each $2\le a\le N$ there exist a critical reset strength
$c_{\mathrm{cr}}^{(a)}$ such that for all $c>c_{\mathrm{cr}}^{(a)}$
avalanches of size greater or equal to $a$ are unstable under return
and avalanches of size smaller than $a$ are stable. For $c\le c_{\mathrm{cr}}^{(N)}$
all avalanches are stable under return. The critical reset strengths
are determined from the equation

\begin{equation}
e^{b\left(1-\left[(N-a)+c_{\mathrm{cr}}^{(a)}(a-1)\right]\varepsilon\right)}=\frac{\left(e^{-bc_{\mathrm{cr}}^{(a)}\varepsilon}-1\right)}{\left(e^{-b\varepsilon}-1\right)}\label{eq: Analyse c crit implicit theorem}\end{equation}
and satisfy $0<c_{\mathrm{cr}}^{(N)}<c_{\mathrm{cr}}^{(N-1)}<\dots<c_{\mathrm{cr}}^{(2)}<1$.\begin{proof}
Since $U_{b}$ is icpd and dcpd, equality holds in \eqref{eq: Analyse Delta_M bounds 2},
i.e. for $\mathcal{F}\in\mathcal{T}_{a_{1}}$\begin{eqnarray}
\Delta M_{\mathcal{F}}\left(\Delta\phi\right):=1-M_{\mathcal{F}}\left(1-\Delta\phi\right) & = & 1-S_{\sigma_{l}}\circ H_{\left(N-a_{1}\right)\varepsilon}\circ J_{a_{1}\varepsilon}\left(1-\Delta\phi\right)\label{eq: Analyse Delta M U_b exact}\end{eqnarray}
Thus the return map for the phase differences only depends on the
avalanche size $a_{1}$ and is independent of the precise form of
the other avalanches $a_{i}$, $i>1$ and intermediate shifts $\sigma_{i}$.
Explicitly \[
\Delta M_{\mathcal{F}}\left(\Delta\phi\right)=\frac{e^{b\varepsilon\left(N-a_{1}+c(a_{1}-1)\right)}}{1-e^{b}}\left(e^{-bc}\left(e^{b}+\left(1-e^{b}\right)\Delta\phi\right)^{c}-1\right)\]
for all $\mathcal{F}\in\mathcal{T}_{a_{1}}$. A straight forward calculation
shows that $\Delta M_{\mathcal{F}}$ has the properties \begin{equation}
\Delta M_{\mathcal{F}}\left(0\right)=0\quad,\quad\frac{d}{d\Delta\phi}\Delta M_{\mathcal{F}}\left(\Delta\phi\right)\ge0\quad\textrm{and}\quad\frac{d^{2}}{d\Delta\phi^{2}}\Delta M_{\mathcal{F}}\left(\Delta\phi\right)\le0\label{eq: Analyse Delta M properties}\end{equation}
Thus if the condition \begin{equation}
\Delta M_{\mathrm{\mathcal{F}}}\left(1-U^{-1}\left(1-\varepsilon\right)\right)\le1-U^{-1}\left(1-\varepsilon\right)\label{eq: Analyse cluster stability 3}\end{equation}
is met all other conditions for $1\le a<a_{1}$ in \eqref{eq: Analyse icdp a-avlanche stable}
are also satisfied. On the other hand almost all perturbations will
cause the avalanche to be triggered by a single oscillator. Thus if
condition \eqref{eq: Analyse cluster stability 3} is not satisfied,
i.e. $\Delta M_{\mathcal{F}}\left(\Delta\phi\right)>\Delta\phi$ for
all $\Delta\phi\ge U^{-1}\left(1-\varepsilon\right)-1$ the avalanche
will split up after a finite number of iterations of the return map.
Thus \eqref{eq: Analyse cluster stability 3} is a necessary and sufficient
condition for stability of an $a$-cluster under the return map. We
are interested in the critical strengths $c_{\mathrm{crit}}^{(a)}$
for which an $a$-cluster becomes unstable and hence we use equality
in \eqref{eq: Analyse cluster stability 3} and basic algebra to obtain
the implicit expressions \eqref{eq: Analyse c crit implicit theorem}
for the $c_{\mathrm{cr}}^{(a)}$.

Since we have assumed $\left(N-1\right)\varepsilon<1$, $b<0$ and
$c\in\left[0,1\right]$ we see that the left hand side of \eqref{eq: Analyse c crit implicit theorem}
lies in the interval $\left(0,1\right)$ and decreases monotonically
with increasing $c$. The right hand side is $0$ for $c=0$ and increases
monotonically with $c$ until it becomes $1$ for $c=1$. Thus by
continuity for all $2\le a\le N$ there always exist a solution $0<c_{\mathrm{cr}}^{(a)}<1$
to this equation. Note that the special case $a=2$ is explicitly
solvable for $c_{\mathrm{cr}}^{(2)}$ and yields \begin{equation}
c_{\mathrm{cr}}^{(2)}=\frac{1}{b\varepsilon}\log\left(1+e^{-b(N-2)\varepsilon+b}\left(1-e^{-b\varepsilon}\right)\right)\label{eq: ccr 2}\end{equation}
For fixed $0\le c<1$ the left hand side of \eqref{eq: Analyse c crit implicit theorem}
is strict monotonically decreasing as $a$ increases whereas the left
hand side is independent of $a$, thus $0<c_{\mathrm{cr}}^{(N)}<c_{\mathrm{cr}}^{(N-1)}<\dots<c_{\mathrm{cr}}^{(2)}<1$.
\end{proof}
\end{prop}
The theoretical prediction \eqref{eq: Analyse c crit implicit theorem}
for the desynchronization transition is plotted in Fig. \ref{fig: Desync prob}
and is in excellent agreement with the numerically observed transition. 
\begin{rem}
Note that \eqref{eq: Analyse c crit implicit theorem} involves all
relevant network parameter. In particular, by choosing $b\rightarrow-\infty$
equation \eqref{eq: ccr 2} shows that $c_{\mathrm{cr}}^{(2)}$ can
be made arbitrarily small. This implies that the entire sequence of
desynchronizing bifurcations may occur over an arbitrary small interval
$\left[c_{\mathrm{cr}}^{(N)},c_{\mathrm{cr}}^{(2)}\right]$. 
\end{rem}

\begin{rem}
We also remark that the number of bifurcation points in this sequence
is $N-1$. At each bifurcation point $c_{\mathrm{cr}}^{(a)}$ all
periodic states with at least one cluster of size $a$ and all other
cluster sizes less or equal to $a$, i.e. an extensive combinatorial
number of states, becomes unstable simultaneously. 
\end{rem}
The mechanism underlying the desynchronization transition are opposing
synchronization and desynchronization dynamics in the network as illustrated
in Fig. \ref{cap: Analyse Convex Synchronization-and-Desynchronization}:
Due to the convexity of the rise function (a) sub-threshold inputs
are always synchronizing and stabilize the avalanche, whereas depending
on the strength of the partial reset supra-threshold inputs in an
avalanche can either (b) synchronize or (c) desynchronize the phases.
Thus for a weak partial reset (e.g. $R_{c}$ with $c\approx0$) states
with large avalanches are stable. When the partial reset is stronger
it desynchronizes the cluster and, depending on the avalanche size,
it may outweigh the synchronization effect due to sub-threshold inputs.
Larger avalanches receive less synchronizing sub-threshold input from
other oscillators and simultaneously produce a larger supra-threshold
input than smaller ones. Thus they lose invariance under return first
when increasing the partial reset strength.

\begin{figure}[t]
\begin{centering}
\includegraphics{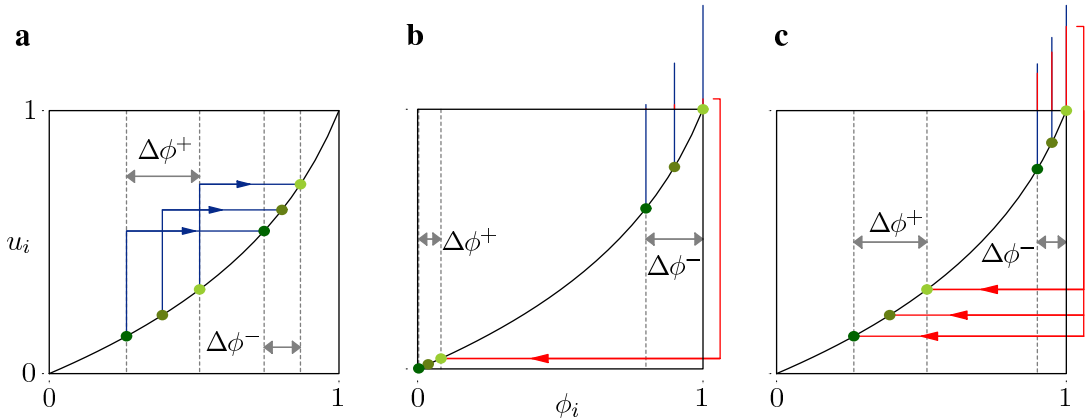} 
\par\end{centering}

\caption{\label{cap: Analyse Convex Synchronization-and-Desynchronization}Synchronization
and Desynchronization of avalanches in networks with convex rise function
and partial reset. \textbf{(a)} Sub-threshold inputs synchronize the
oscillators. The phase difference of a cluster before pulse reception
$\Delta\phi^{+}$ is decreased to $\Delta\phi^{-}$ afterwards, i.e.
$\Delta\phi^{+}<\Delta\phi^{-}$. \textbf{(a)} Weak partial reset
(e.g. $c\approx0$ for $R_{c}$) synchronize phase differences: $\Delta\phi^{+}<\Delta\phi^{-}$.
\textbf{(b)} Due to the convexity of the rise function a strong partial
reset ($c\approx1$) expands the phase differences $\Delta\phi^{+}>\Delta\phi^{-}$.
Clusters lose stability if mechanism in (c) becomes dominant over
the stabilizing effect (a).}

\end{figure}

\section{\label{sec:Robustness}Robustness of the Desynchronization Transition}

The desynchronization transition is robust against structural perturbations
in the coupling matrix and the rise function $U$.

\subsection{Coupling Strength Inhomogeneity}

\begin{figure}[th]
\begin{centering}
\includegraphics{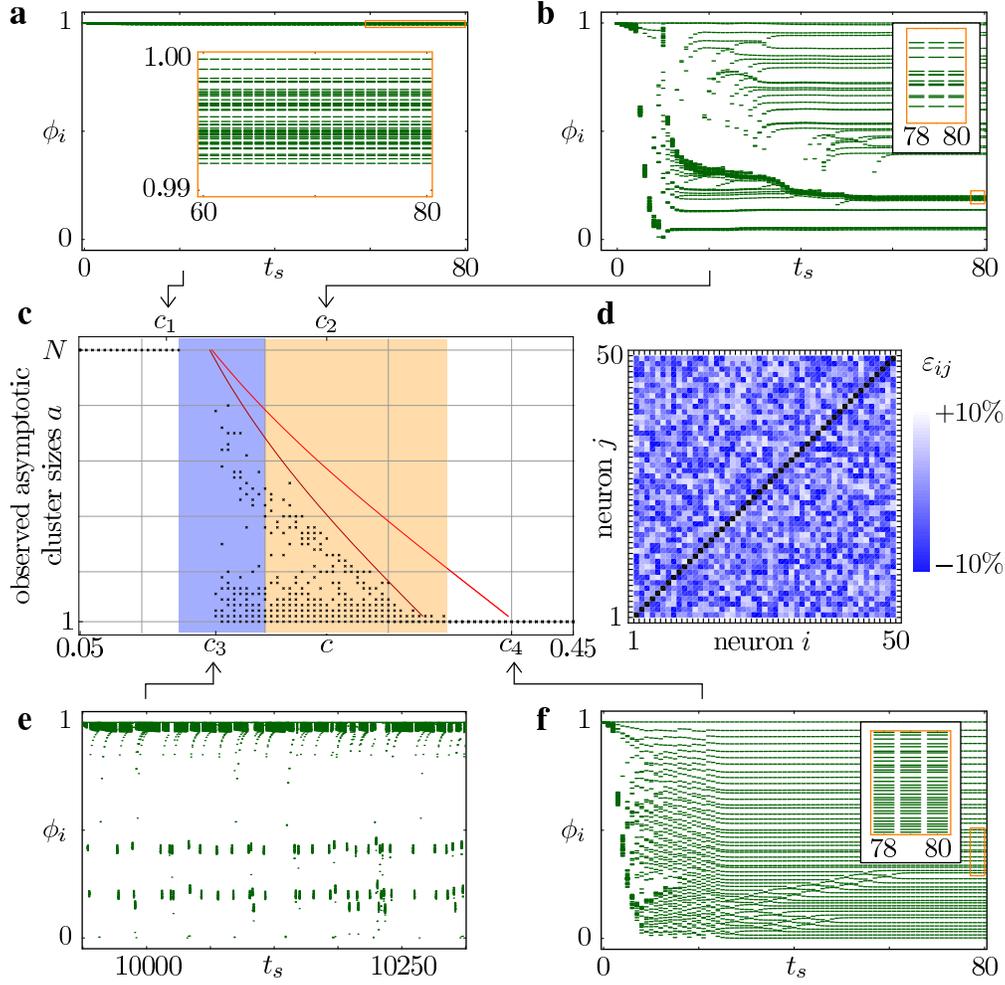}
\par\end{centering}

\caption{\label{cap: Analyse desynchronization-transition inhomogen}Sequential
desynchronization transition in an inhomogeneous network. Shown are
the dynamics of a network of $N=50$ units with convex rise-function
$U_{\mathrm{LIF}}^{\mathrm{CB}}$ ($E_{\mathrm{eq}}=1.1$ and $E_{\mathrm{syn}}=3$)
and linear partial reset $R_{c}$. \textbf{(d)} coupling strengths
are drawn from a random uniform distribution from $\varepsilon_{ij}\in\left[0.009,0.011\right]$
excluding self-coupling ($\varepsilon_{ii}=0$). Simulations are initialized
in a perturbed synchronous state. \textbf{(a)} Synchronous firing
for $c_{1}=0.12$. Inset: Phase are not synchronized due to the coupling
inhomogeneity. \textbf{(b)} For $c_{2}=0.19$ the synchronous state
becomes unstable and smaller avalanches are observed in the asymptotic
dynamics Inset: As in the synchronous state phases of units in a single
avalanche are not phase synchronized. \textbf{(e)} For $c_{3}=0.16$
we do observe aperiodic firing until $t=15000$. \textbf{(f)} For
$c_{4}=0.4$ asynchronous firing is observed in the asymptotic dynamics.
Inset: Spacings of the phases are irregular due to the network heterogeneity.
\textbf{(c)} Asymptotic cluster sizes of periodic states for different
$c$ values starting from $100$ different perturbed synchronous states.
The lines show the lower and upper estimates for the transition obtained
from theorem \ref{thm: Analyse Avalanche stability icpd dcpd} assuming
a homogeneous network with coupling strength $\varepsilon=0.01$.
The shaded blue area indicates that states with aperiodic dynamics
until $t=15000$ are observed. }

\end{figure}

The desynchronization transition is robust against perturbations in
the coupling matrix $\varepsilon_{ij}$. Our numerical experiments
show that the transition is also observed when using coupling strengths
from a uniform distribution on an interval $[\varepsilon_{\mathrm{min}},\varepsilon_{\mathrm{max}}]$
for a interval length $\Delta\varepsilon=\varepsilon_{\mathrm{max}}-\mathrm{\varepsilon_{min}}$
as large as $20\%$ of the average coupling strength $\bar{\varepsilon}=\left(\varepsilon_{\mathrm{max}}-\mathrm{\varepsilon_{min}}\right)/2$
(cf. fig. \ref{cap: Analyse desynchronization-transition inhomogen}).
When $\Delta\varepsilon$ becomes larger usually complex spike patterns
and non-periodic states are observed.

In inhomogeneous networks sub-threshold inputs of different strengths
desynchronize units initially at the same phase. Thus coupling inhomogeneity
destabilizes clusters of a given size $a$. In fact, already the lower
bound $c_{\mathrm{crit}}^{(a)}$ obtained for homogeneous networks
via theorem \ref{thm: Analyse Avalanche stability icpd dcpd} using
the coupling strength $\bar{\varepsilon}$ over-estimates the stability
of the clusters as shown in fig. \ref{cap: Analyse desynchronization-transition inhomogen}.
The regime where we observe aperiodic dynamics becomes larger in comparison
to homogeneous networks with the same average coupling strength (e.g.
compare fig. \ref{cap: Analyse desynchronization-transition inhomogen}
and fig. \ref{cap: Analyse Bounds-for-clustersizes}). This is due
to cluster states in the homogeneous network with asymptotic phases
of the clusters which are close to an absorption (i.e. there are $\sigma_{i}^{*}\approx0$
for some $i$). A perturbation in the coupling now enables the absorption
and the restless competition between desynchronization and synchronization
(cf. sec. \ref{sub:Stability-against-Cluster}) induces the aperiodic
dynamics.

Another effect of inhomogeneous coupling is that units synchronized
to fire together in clusters are not phase synchronized (cf. insets
in fig. \ref{cap: Analyse desynchronization-transition inhomogen})
in the asymptotic dynamics. Also intra-cluster phases of equally sized
clusters and in particular of the asynchronous state show irregular
spacings.

\subsubsection{Sigmoidal Rise Functions}

\begin{figure}[th]
\begin{centering}
\includegraphics{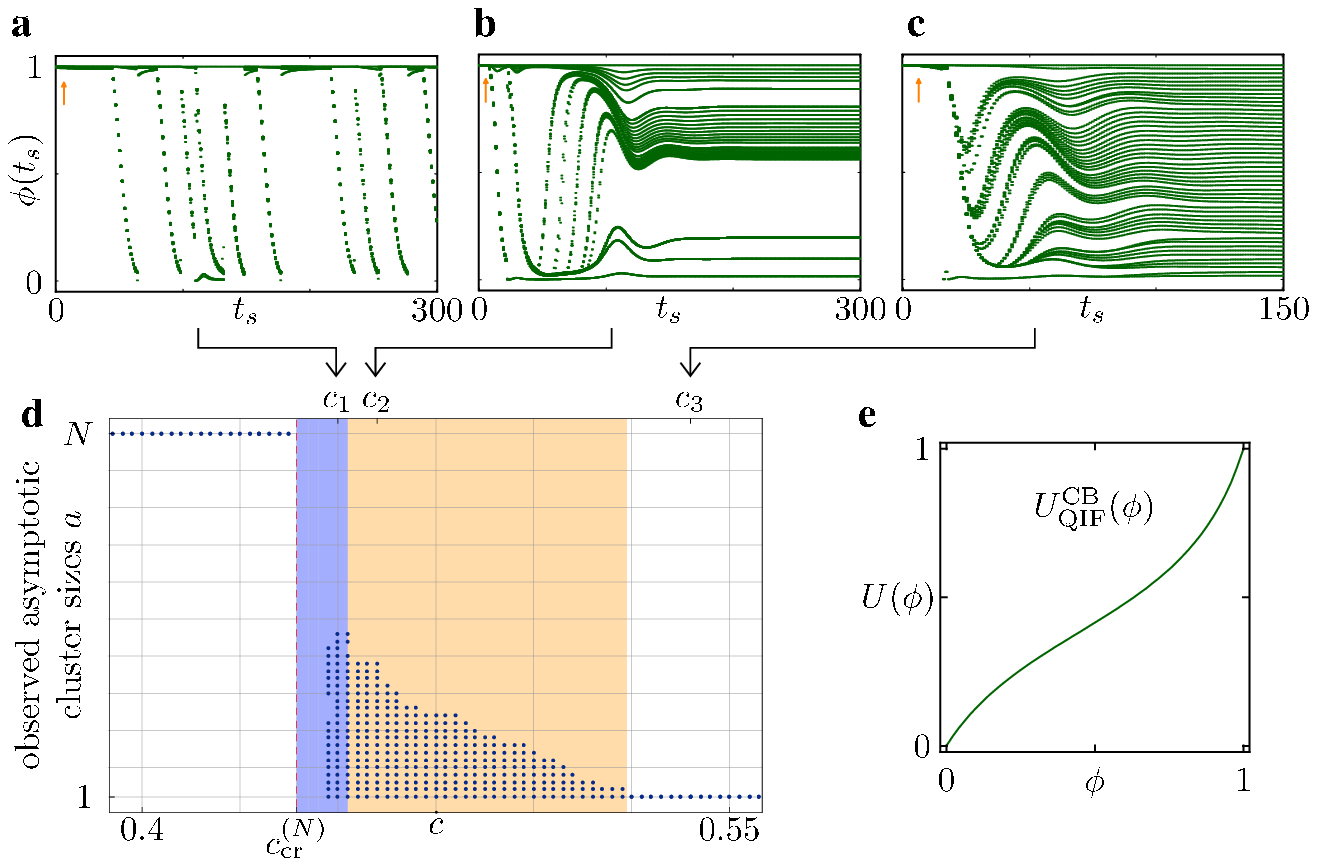}
\par\end{centering}

\caption{\label{cap: Analyse desync general RF}Sequential desynchronization
transition in networks of neural oscillators with a sigmoidal rise
function. Shown are the dynamics of a homogeneous network ($N=100$,
$\varepsilon=0.002$) with linear partial reset $R_{c}$ and \textbf{(e)}
sigmoidal rise function $U_{\mathrm{LIF}}^{\mathrm{CB}}$ ($E_{\mathrm{syn}}=2$,
$\alpha=-1$, $\beta=1$). Starting with synchrony and inducing a
small perturbation (arrow) the network shows \textbf{(a)} aperiodic
dynamics for $c=c_{1}=0.45$, \textbf{(b)} clustering for $c=c_{2}=0.46$
and \textbf{(c)} asynchronous dynamics for $c=c_{3}=0.54$. Note the
oscillations of the phase which do not appear for purely convex rise
functions (cf. fig. \ref{cap: Analyse Bounds-for-clustersizes}).
\textbf{(d) }Cluster sizes of periodic states observed in the dynamics
at $t=5000$ starting from $200$ perturbed synchronous states for
each value of $c$. Shaded area marks the transition region with states
other than solely synchronous or asynchronous. The blue shaded are
marks occurrence of aperiodic dynamics. The dashed line indicates
the critical $c_{\mathrm{crit}}^{(N)}$ determined from \eqref{eq: Analyse Cluster instability}
for $a_{1}=N$ above which synchronous firing becomes unstable. }

\end{figure}

Typically rise functions in biological or physical systems are neither
purely concave nor purely convex. In particular intrinsic neuronal
dynamics are often best described with a sigmoidal rise function.
The quadratic-integrate-and-fire or exponential-integrate-and-fire
neuron \cite{ErmentroutKopell:1986,FourcaudTrocmeHanselVreeswijkBrunel:2003}
(cf. also Appendix \ref{sec:Rise-Functions}) constitute major examples.
In networks with sigmoidal rise functions a combination of the effects
inherent to concave and convex rise functions influences the network
dynamics: Synchronization of units to larger clusters due to the concave
part (cf. \cite{Mirollo:1990,KirstTimme:2009}) and stabilization
of states with asynchronously firing clusters due to the convex part
(cf. theorem \ref{thm: Analyse cluster states}). Using strictly neuronal
partial resets numerical studies show that for rise functions with
dominant concave part synchronized firing of oscillators in the asymptotic
state is typically found. In contrast, if the convex part is larger
it is more likely to find clusters of smaller sizes and the asynchronous
state. Indeed, for general rise functions $U$ we still obtain the
stability matrix $\mathbf{A}$ in \eqref{eq: Stability Matrix} but
the non-zero entries \eqref{eq: aij} can become larger than $1$
in the regime where $U$ is concave. Thus if the concave part becomes
dominant the eigenvalues are no longer bounded by $1$ and asynchronous
clusters states become unstable.

In fig. \ref{cap: Analyse desync general RF} a desynchronization
transition for the sigmoidal rise function $U_{\mathrm{QIF}}^{\mathrm{CB}}$
and linear partial reset $R_{c}$ is shown.  In the synchronous state
oscillators do not receive any intermediate sub-threshold pulses between
successive firing and the return map for an oscillator with phase
$\phi$ can be written as \[
M_{\left\{ \left\{ 1,\dots,N\right\} ,\sigma\right\} }(\phi)=U^{-1}\left(R\left(U(\phi)+(N-1)\varepsilon-1\right)\right)+1-U^{-1}\left(R\left((N-1)\varepsilon\right)\right)\]
for any partial reset $R$ and any rise function $U$. After a perturbation
the avalanche is typically triggered by a single unit and thus the
synchronous state becomes unstable if $M_{\left\{ \left\{ 1,\dots,N\right\} ,\sigma\right\} }(\phi)<\phi$
for all $\phi\in\left[1-U^{-1}\left(1-\varepsilon\right),1\right]$
which yields the condition \eqref{eq: Analyse Cluster instability}
for $a_{1}=N$. This can be used to determine the onset of a desynchronization
transition in the general case as shown in fig. \ref{cap: Analyse desync general RF}
(dashed line). The stability of smaller avalanches $a_{1}<N$ can
still be estimated with the help of theorem \ref{thm: Analyse Avalanche stability icpd dcpd}
if the rise function is dcpd but not necessarily convex. Conditions
for the sigmoidal rise functions $U_{\mathrm{QIF}}$ and $U_{\mathrm{QIF}}^{\mathrm{CB}}$
to be dcpd are given in appendix \ref{sec:Rise-Functions}.

Desynchronization due to a partial reset has three components: Translation
of phase differences into potential differences via the rise function
$U$, the relative change of potential differences due to the partial
reset $R$ after supra-threshold excitation and back-translation of
this potential difference into phase differences via $U^{-1}$ (cf.
fig. \ref{cap: Analyse Convex Synchronization-and-Desynchronization}c).
For convex rise functions the slope in the reset zone $I^{R}=\left[0,U^{-1}\left(R\left((N-1)\varepsilon\right)\right)\right]$
is always smaller than in the supra-threshold zone $I^{T}=\left[U^{-1}\left(1-(N-1)\varepsilon\right),1\right]$.
As a consequence the phase differences in $I^{T}$ are translated
via $U$ to larger potential differences and the potential differences
after reset become larger phase differences during the back translation
$U^{-1}$. This causes an effective phase desynchronization even for
partial resets that are non-expansive as depicted in fig. \ref{cap: Analyse Convex Synchronization-and-Desynchronization}c. 

For general rise functions and non-expansive partial resets the destabilization
of a cluster state due to a partial reset thus can only occur if the
slopes in $I^{T}$ are sufficiently larger than in $I^{R}$. In fact
if this ratio becomes to small the transition may not be observed
completely for non-expansive partial reset, e.g. for $R_{c}$ in the
range $c\in\left[0,1\right]$ and can be shifted to partial resets
that have to be expansive (e.g. for $c>1$) .

Finally note that, in contrast to convex rise functions, for sigmoidal
rise functions we always observe ''damped oscillations'' in the
Poincar\'e phase plots fig. \ref{cap: Analyse desync general RF}
(b,c). The amplitudes of these oscillations become larger when the
slope of the rise function at the point of inflection becomes smaller.
We therefore attribute these oscillations to sub-threshold inputs
received by oscillators near the inflection point of the rise function.

\section{\label{sec:Discussion}Discussion}

In summary, we proposed a model of pulse-coupled threshold units with
partial reset. This partial reset, an intrinsic response property
of the local units, acts as a desynchronization mechanism in the collective
network dynamics. It causes an extensive sequence of desynchronizing
bifurcations of cluster states networks of pulse-coupled oscillators
with convex rise function. This sequential desynchronization transition
is robust against structural perturbations in the coupling strength
and variations of the local subthreshold dynamics.

Previous studies have not particularly focused on the collective implications
of partial reset or similar graded resets. In network models with
pulses that are extended in time typically a full conservation of
the input is considered \cite{Tsodyks:1993,vanVreeswijk:1994,HanselMato:2001a}.
Models with instantaneous responses to inputs consider fully dissipative
reset ($R\left(\zeta\right)\equiv0$ in our model) \cite{Mirollo:1990,Gerstner:1993,Bottani:1995,SennUrbanczik:2001,TimmeWolfGeisel:2002a,TimmeWolf:2008},
fully conservative reset ($R\left(\zeta\right)=\zeta$) \cite{Bressloff:1997a,BressloffCoombes:2000}
as well as both extremes \cite{Hopfield:1995b} without discussing
particular consequences of the reset mechanism. Here we closed this
gap and showed that in fact the reset mechanism plays an important
role in synchronization processes.

Partial reset in pulse-coupled oscillators keeps the collective network
dynamics analytically tractable and at the same time describes additional,
physically or biologically relevant dynamical features of local units.
In neurons, for instance, synaptic inputs are collected in the dendrite
and then transmitted to the cell body (soma). At the soma the integration
of the membrane potential takes place and spikes are generated. Remaining
input charges on the dendrite not used to trigger a spike at the soma
may therefore contribute to the potential after somatic reset \cite{DorionOswald:2007,Rospars:1993,Bressloff:1995}. 

Such features are effectively modeled by the simple partial reset
introduced here. In particular, spike time response curves (that may
be obtained for any tonically firing neuron \cite{Perkel:1964,ReyesFetz:1993,OprisanCanavier:2002,GalanErmentroutUrban:2005})
encode the shortening of the inter-spike intervals (ISI) following
an excitatory input at different phases of the neural oscillation.
An excitatory stimulus that causes the neuron to spike will maximally
shorten the ISI in which the stimulus is applied. Additionally, the
second ISI that follows is typically affected as well, e.g. due to
compartmental effects. Exactly this shortening of the second ISI is
characterized by appropriately choosing a partial reset function in
our simplified system. The details of such a description and consequences
for networks of more complicated neuron models are studied separately
\cite{KirstTimme:2009}. For instance, networks of two-compartment
conductance based neurons indeed exhibit similar desynchronization
transitions when varying the coupling between soma and dendrite \cite{KirstTimme:2009}
which in our simplified model controls the partial reset.

The desynchronization due to the partial reset, i.e. due to local
processing of supra-threshold input, differs strongly from that induced
by previously known mechanisms based on, e.g. heterogeneity, noise,
or delayed feedback \cite{vanVreeswijk:1994,vanVreeswijk:2000,Maistrenko:2004,Kiss:2007,Popovych:2003,Denker:2004}.
Possibly, this desynchronization mechanism may also be helpful in
modified form to prevent synchronization in neural activity such as
in Parkinson tremor or in epileptic seizures \cite{Tass:2003,Tass:2006}.

In this work we developed a partial reset for supra-threshold inputs
and considered purely homogeneous and globally excitatory coupled
systems. For inhibitory couplings one can define a lower threshold
\cite{Denker:2002} below which inhibitory inputs becomes less effective,
i.e. a partial inhibition. In models of neurons, for instance, this
could characterize shunting inhibition \cite{AlgerNicoll:1979}. If
two units simultaneously receive inhibitory inputs below a lower threshold,
a zero partial inhibition, i.e. setting the state of the units to
a fixed lower value, is strongly synchronizing in analogy to a full
reset after supra-threshold excitation. Our findings suggest that
similar to a partial reset a less synchronizing non-zero partial inhibition
may also have a strong influence on the collective network dynamics.
In biologically more detailed neuronal network models both excitatory
and inhibitory couplings as well as complex network topologies play
important roles in generating irregular \cite{vanVreeswijk:1996a}
and synchronized spiking dynamics \cite{Abeles:2004}. An interesting
task would therefore be to study the impact of partial resets in such
networks.\\

This work was supported by the Federal Ministry for Education and
Research (BMBF) by a grant number 019Q0430 to the Bernstein Center
for Computational Neuroscience (BCCN) Göttingen and by a grant of
the Max Planck Society to MT.

\appendix

\section{The Eneström-Kakeya Theorem\label{sec: Appendix EnestromKakeya}}

\subsection{\label{sec: Spectral-Radius}Spectral-Radius and Matrix-Norm}

Let $\mathbf{A}=a_{ij}$ be a $n\times n$ matrix. The spectral radius
$\rho$ of a $\mathbf{A}$ is defined as \cite{Horn:1996} \begin{equation}
\rho(\mathbf{A})=\max_{\left\Vert \mathbf{x}\right\Vert =1}\left\Vert \mathbf{A}\mathbf{x}\right\Vert =\max_{i=1,..,n}\left|\lambda_{i}\right|\label{eq: specral radius}\end{equation}
 where $\left\Vert \right\Vert $ denotes a norm and $\{\lambda_{i}\}_{i=1}^{n}$
are the complex eigenvalues of $\mathbf{A}$.If $\left\Vert -\right\Vert $
is any matrix norm (see \cite{Mehta:1989}) the inequality \begin{equation}
\rho(A)\le\left\Vert A\right\Vert \label{eq: Appendix spectral radius < matrix norm}\end{equation}
is valid and in fact $\rho(A)=\inf\left\Vert A\right\Vert $ where
the infimum is taken over all matrix norms \cite{Horn:1996}. Here
we only need the \emph{maximum-absolute-column-sum norm} of $\mathbf{A}$
defined as\begin{equation}
\left\Vert \mathbf{A}\right\Vert =\max_{j=1,...,n}\sum_{i=1}^{n}\left|a_{ij}\right|\label{eq: Appendix max column matrix norm}\end{equation}

\subsection{\label{sub: Appendix Companion-Matrices}Companion Matrices}

A $(n+1)\times(n+1)$ companion matrix $\mathbf{C}$ has the standard
form \begin{equation}
\mathbf{C}=\left(\begin{array}{cccc}
0 & \dots & 0 & -\tilde{c}_{0}\\
1 &  & 0 & -\tilde{c}_{1}\\
 & \ddots &  & \vdots\\
0 &  & 1 & -\tilde{c}_{n}\end{array}\right)\label{eq: Appendix companion matrix}\end{equation}
 with characteristic polynomial\begin{equation}
\tilde{p}_{n+1}(z)=\det\left(z-\mathbf{C}\right)=\tilde{c}_{0}+\tilde{c}_{1}z+...+\tilde{c}_{n}z^{n}+z^{n+1}\label{eq: Appendix charactersitic polynomial}\end{equation}

\subsection{\label{sec: Appendix The-Enestrm-Kakeya-Theorem}The Eneström-Kakeya
Theorem}

The Eneström-Kakeya theorem%
\footnote{In 1893 the Swedish actuary and mathematics historian Gustaf Eneström
published this result of roots of certain polynomials with real coefficients
in a paper on pension insurance (in Swedish) \cite{Enestroem:1893}.
This result is now often called the Eneström-Kakeya theorem, since
S. Kakeya published a similar result in 1912-1913 \cite{Kakeya:1912}.
But Kakeya's theorem contained a mistake, which was corrected by A.
Hurwitz in 1913 \cite{Hurwitz:1933}.%
} \cite{Enestroem:1893,Kakeya:1912,Hurwitz:1933,AndersonVarga:1972,Horn:1996}
can be stated in the following form 
\begin{thm}
Let $p_{n}(z)=\sum_{j=0}^{n}c_{j}z^{j}$ with $c_{j}>0$ then for
all $\lambda$ with $p_{n}\left(\lambda\right)=0$ \[
\left|\lambda\right|\le\max_{0\le i<n}\left\{ \frac{c_{i}}{c_{i+1}}\right\} =:\beta\]
 \begin{proof}
Note first that $\beta>0$. We set\begin{equation}
\tilde{p}_{n+1}(z):=\frac{(z-1)p_{n}(\beta z)}{c_{n}\beta^{n}}=z^{n+1}+\sum_{i=0}^{n}\tilde{c}_{i}z^{i}\label{eq: Appendix ernestoem polynomial}\end{equation}
 where \[
\tilde{c}_{i}=\begin{cases}
\frac{c_{i-1}-\beta c_{i}}{c_{n}\beta^{n-i+1}} & \quad\quad1\le i\le n\\
\frac{-c_{0}}{c_{n}\beta^{n}} & \quad\quad i=0\end{cases}\]
Using the definition of $\beta$ one observes that $\tilde{c}_{j}\le0$.
Comparing \eqref{eq: Appendix charactersitic polynomial} with \eqref{eq: Appendix ernestoem polynomial}
the companion matrix of $\tilde{p}_{n+1}$ is given by \eqref{eq: Appendix companion matrix}.
Since $1+\sum_{j=1}^{n+1}\tilde{c}_{j}=\tilde{p}_{n+1}(1)=0$ if follows
that $\left\Vert \mathbf{C}\right\Vert =\sum_{j=1}^{n+1}\left|\tilde{c}_{j}\right|=-\sum_{j=1}^{n+1}\tilde{c}_{j}=1$
when using the maximum-absolute-column-sum norm \eqref{eq: Appendix max column matrix norm}
and hence from \eqref{eq: Appendix spectral radius < matrix norm}\[
\rho(\mathbf{C})\le1\]
Thus for all $\tilde{\lambda}$ with $p_{n+1}\left(\tilde{\lambda}\right)=0$
we have $\left|\tilde{\lambda}\right|\le\rho(\mathbf{C})\le1$. For
a $\lambda$ with $p_{n}\left(\lambda\right)=0$ it follows from the
definition of $\tilde{p}_{n+1}$ that $\tilde{p}_{n+1}\left(\tilde{\lambda}\right)=0$
for $\tilde{\lambda}=\frac{\lambda}{\beta}$ and thus $\left|\lambda\right|\le\beta$.\end{proof}
\begin{cor}
Let $\mathbf{A}$ be a matrix of the form (cf. \eqref{eq: Stability Matrix}) 

\[
\mathbf{A}=\left(\begin{array}{ccccc}
-a_{n} & a_{1} & 0 & \dots & 0\\
-a_{n} & 0 & a_{2} & \ddots & \vdots\\
\vdots & \vdots & \ddots & \ddots & 0\\
-a_{n} & 0 & \dots & 0 & a_{n-1}\\
-a_{n} & 0 & \dots & 0 & 0\end{array}\right)\]
with $a_{i}>0$ then \[
\rho(\mathbf{A})\le\max\left\{ a_{i}\right\} _{i=1}^{n}\]
\end{cor}
\begin{proof}
By a permutation of rows and columns we can cast $\mathbf{A}$ into
a matrix $\mathbf{B}=b_{i,j}$ with non-zero entries $b_{i,(i+1)}=a_{i}$,
$i\in\left\{ 1,\dots,n-1\right\} $ and $b_{i,n}=-a_{n}$, $i\in\left\{ 1,\dots,n\right\} $.
This does not change the spectral radius. The similarity transformation
to $\mathbf{C}=\mathbf{Q}^{-1}\mathbf{B}\mathbf{Q}$ with $\mathbf{Q}=\mathrm{diag}\left(q_{1},\dots,q_{N-1}\right)$
and $q_{1}=1$, $q_{i}=\prod_{j=1}^{i-1}a_{j}$, $i\in\left\{ 2,\dots,n\right\} $
also preserves the spectral radius and $\mathbf{C}$ has the form
of a companion matrix \eqref{eq: Appendix companion matrix} with
$c_{i}=\prod_{j=i+1}^{n}a_{i}>0$, $i\in\left\{ 0,\dots,n-1\right\} $.
Thus $\rho(\mathbf{A})=\rho(\mathbf{C})\le\max_{0\le i<n}\left\{ \frac{c_{i}}{c_{i+1}}\right\} =\max_{1\le i\le n}\left\{ a_{i}\right\} $ 
\end{proof}

\section{\label{sec:Rise-Functions}Rise Functions}

\end{thm}

\subsection{\label{sub: IF Rise-Functions}Rise Functions for Integrate-and-Fire
Models}

In this section we derive the rise functions for single variable models
of the form \[
\frac{d}{dt}v=F\left(v\right)+I_{\mathrm{in}}(t)\]
We distinguish between potential independent inputs $I_{\mathrm{in}}(t)=P(t)$
with $P(t)=\sum_{s}\varepsilon_{s}\delta\left(t-t_{s}\right)$ and
the conductance based approach $I_{\mathrm{in}}\left(t\right)=g_{\mathrm{syn}}P(t)\left(E_{\mathrm{syn}}-v(t)\right)$,
$E_{\mathrm{syn}}>1$. More generally, if $I_{\mathrm{in}}(t)=Q(v(t))P(t)$
and $Q(t)>0$ the transformation \begin{equation}
u(t)=\frac{1}{M}\int_{0}^{v(t)}\frac{1}{Q(v)}\mathrm{d}v\quad\quad M=\int_{0}^{1}\frac{1}{Q(v)}\mathrm{d}v\label{eq: Model transform A away}\end{equation}
yields\begin{equation}
\frac{d}{dt}u=\hat{F}\left(u\right)+\frac{1}{M}P(t)\quad\quad\hat{F}\left(u\right)=\frac{1}{M}\frac{F\left(v\left(u\right)\right)}{Q(v(u))}\label{eq: Model transform A away  tilde_F}\end{equation}
i.e. a potential independent input $I_{\mathrm{in}}$. Thus if the
rise function $U$ for $I_{\mathrm{in}}(t)=P(t)$ is known the conductance
based rise function $U^{\mathrm{CB}}$ is calculated with the help
of \eqref{eq: Model Oscillators Transform} as\begin{equation}
U^{\mathrm{CB}}\left(\phi\right)=\frac{\ln\left(1-E_{\mathrm{syn}}^{-1}U\left(\phi\right)\right)}{\ln\left(1-E_{\mathrm{syn}}^{-1}\right)}\label{eq: Model U_CB}\end{equation}

The leaky-integrate-and-fire (LIF) model \cite{Lapicque:1907} is
given by $F(u)=-g_{l}u+I_{\mathrm{ext}}$ which yields\begin{equation}
U_{\mathrm{LIF}}\left(\phi\right)=E_{\mathrm{eq}}\left(1-e^{-g_{\mathrm{l}}T_{\mathrm{LIF}}\phi}\right)\label{eq: Model U leaky IF}\end{equation}
where $T_{\mathrm{LIF}}=-\frac{1}{g_{\mathrm{l}}}\ln\left(1-E_{\mathrm{eq}}\right)$
and $E_{\mathrm{eq}}=\frac{I_{\mathrm{ext}}}{g_{l}}+E_{\mathrm{l}}>1$.
This yields\begin{equation}
U_{\mathrm{LIF}}^{\mathrm{CB}}\left(\phi\right)=\frac{\ln\left(1-E_{\mathrm{syn}}^{-1}U_{\mathrm{LIF}}\left(\phi\right)\right)}{\ln\left(1-E_{\mathrm{syn}}^{-1}\right)}\label{eq: Model U CB LIF}\end{equation}

For the quadratic-integrate-and-fire (QIF) model \cite{FourcaudTrocmeHanselVreeswijkBrunel:2003}
with $F(u)=g_{2}\left(E_{\mathrm{r}}-u\right)\left(E_{\mathrm{t}}-u\right)+I_{\mathrm{ext}}$
one obtains for $I_{\mathrm{syn}}(t)=P(t)$ \begin{equation}
U_{\mathrm{QIF}}\left(\phi\right)=\frac{\alpha-\tan\left(\arctan\left(\alpha\right)-\phi\left(\arctan\left(\alpha\right)-\arctan\left(\beta\right)\right)\right)}{\alpha-\beta}\label{eq: Model U QIF def}\end{equation}
where $\alpha=\frac{E_{r}+E_{t}}{\gamma}\text{ , }\quad\beta=\alpha-\frac{2}{\gamma},\text{ \ensuremath{\quad}}\gamma=\sqrt{\frac{4I_{\mathrm{ext}}}{g_{2}}-\left(E_{\mathrm{t}}-E_{\mathrm{r}}\right)^{2}}>0$.
Hence \begin{equation}
U_{\mathrm{QIF}}^{\mathrm{CB}}\left(\phi\right)=\frac{\ln\left(1-E_{\mathrm{syn}}^{-1}U_{\mathrm{QIF}}\left(\phi\right)\right)}{\ln\left(1-E_{\mathrm{syn}}^{-1}\right)}\label{eq: Model U CB QIF}\end{equation}

Note that depending on the IF model and coupling type convex, concave
and sigmoidal shapes are possible (cf. tab \ref{cap: Analyse Properties-of-rise-functions}).
We remark that as $E_{\mathrm{syn}}\rightarrow\infty$ we recover
the potential independent model from the conductance based version,
i.e. $U^{\mathrm{CB}}\rightarrow U$ and the conditions for the different
properties of $U^{\mathrm{CB}}$ become the conditions for $U$ in
tab. \ref{cap: Analyse Properties-of-rise-functions}.

\subsection{\label{sub:Icpd-and-Dcpd}Icpd and Dcpd Rise Functions}

\begin{table}[th]
\begin{centering}
\begin{tabular}{|c||>{\centering}m{2.5cm}|c|m{2.1cm}|m{2cm}|c|m{3.5cm}|}
\hline 
$U$ & parameter domain & concave & \centering{}convex  & \centering{}sigmoidal & icpd & \centering{}dcpd\tabularnewline
\hline
\hline 
$U_{\mathrm{LIF}}$ & \centering{}$E_{\mathrm{eq}}>1$ & $\surd$ & \centering{}- & \centering{}- & $\surd$ & \centering{}-\tabularnewline
\hline 
$U_{\mathrm{LIF}}^{\mathrm{CB}}$ & $E_{\mathrm{syn}}>1$, $E_{\mathrm{eq}}>1$ & $E_{\mathrm{syn}}>E_{\mathrm{eq}}$ & \centering{}$E_{\mathrm{syn}}<E_{\mathrm{eq}}$ & \centering{}- & $E_{\mathrm{syn}}\ge E_{\mathrm{eq}}$ & \centering{}$E_{\mathrm{syn}}\le E_{\mathrm{eq}}$\tabularnewline
\hline 
$U_{\mathrm{QIF}}$ & $0\le\alpha<\infty$, $-\infty<\beta\le0$,

$\alpha>\beta$ & $\beta=0$ & \centering{}$\alpha=0$ & \centering{}$\beta<0<\alpha$ & - & \begin{centering}
$\alpha\le1$
\par\end{centering}

\centering{}$-1\le\beta$ \tabularnewline
\hline 
$U_{\mathrm{QIF}}^{\mathrm{CB}}$ & $E_{\mathrm{syn}}>1$, $0\le\alpha<\infty$, $-\infty<\beta\le0$  & - & \begin{centering}
$0\le1+$
\par\end{centering}

\centering{}$\alpha\left(\alpha-2\eta\right)$ & \begin{centering}
$0>1+$
\par\end{centering}

\centering{}$\alpha\left(\alpha-2\eta\right)$ & \centering{}- & \centering{}$\alpha^{2}\le\frac{\eta}{\eta-\alpha-\alpha^{-1}}$ $\beta^{2}\le\frac{\eta-\alpha+\beta}{\eta-\alpha-\beta^{-1}}$ \tabularnewline
\hline 
$U_{b}$ & $b\in\mathbb{R}\setminus\left\{ 0\right\} $ & $b<0$ & \centering{}$b>0$ & \centering{}- & $\surd$ & \centering{}$\surd$\tabularnewline
\hline
\end{tabular}\vspace{3mm}\\

\par\end{centering}

\caption{\label{cap: Analyse Properties-of-rise-functions}Properties of different
rise functions. $\eta=E_{\mathrm{syn}}\left(\alpha-\beta\right)$.
}

\end{table}
Usually it is difficult to verify the icpd or dcpd property \eqref{def:icpd dcpd}
of a rise function. Here we show that it is closely related to the
third derivative of $U$. 

We first note that $\Delta H$ obeys the relations $\Delta H\left(\phi,0,\varepsilon\right)\equiv0$
and $\Delta H\left(\phi,\Delta\phi,0\right)\equiv\Delta\phi$ and
hence $\frac{\partial}{\partial\phi}\Delta H\left(\phi,\Delta\phi,0\right)=0$
and \[
\frac{\partial}{\partial\phi}\Delta H\left(\phi,\Delta\phi,\varepsilon\right)=\int_{0}^{\varepsilon}\int_{0}^{\Delta\phi}\frac{\partial}{\partial\phi}\frac{\partial}{\partial\varepsilon}\frac{\partial}{\partial\Delta\phi}\Delta H\left(\phi,\tilde{\Delta\phi},\tilde{\varepsilon}\right)\mathrm{d\tilde{\Delta\phi}\mathrm{d}\tilde{\varepsilon}}\]
Thus $U$ is icpd if\begin{equation}
\frac{\partial^{3}}{\partial\phi\partial\varepsilon\partial\Delta\phi}\Delta H\left(\phi,\Delta\phi,\varepsilon\right)\ge0\quad\text{for all}\,\left(\phi,\Delta\phi,\varepsilon\right)\in\mathcal{D}\label{eq: Analyse partial deriv 3 Delta H}\end{equation}
Using $\le$ instead of $\ge$ yields an analogous condition for dcpd
$U$. By definition of $\Delta H$ eq. \eqref{eq: Analyse partial deriv 3 Delta H}
yields the condition\begin{eqnarray*}
\frac{\partial^{3}}{\partial\phi\partial\varepsilon\partial\Delta\phi}\Delta H\left(\phi,\Delta\phi,\varepsilon\right) & = & 3\frac{U''\left(H\left(\phi+\Delta\phi,\varepsilon\right)\right)^{2}U'\left(\phi+\Delta\phi\right)^{2}}{U'\left(H\left(\phi+\Delta\phi,\varepsilon\right)\right)^{5}}\\
 &  & -\frac{U''\left(\phi+\Delta\phi\right)U''\left(H\left(\phi+\Delta\phi,\varepsilon\right)\right)}{U'\left(H\left(\phi+\Delta\phi,\varepsilon\right)\right)^{3}}-\frac{U'\left(\phi+\Delta\phi\right)^{2}U'''\left(H\left(\phi+\Delta\phi,\varepsilon\right)\right)}{U'\left(H\left(\phi+\Delta\phi,\varepsilon\right)\right)^{4}}\\
 & \ge & 0\quad\forall\,\left(\phi,\Delta\phi,\varepsilon\right)\in\mathcal{D}\,.\end{eqnarray*}
Substituting $H\left(\phi+\Delta\phi,\varepsilon\right)\rightarrow\phi$
and $\phi+\Delta\phi\rightarrow\psi$ one obtains \begin{equation}
U'''\left(\phi\right)\le3\frac{U''\left(\phi\right)^{2}}{U'\left(\phi\right)}-\frac{U''\left(\psi\right)U''\left(\phi\right)U'\left(\phi\right)}{U'\left(\psi\right)^{2}}\quad\forall\,0\le\psi\le\phi\le1\label{eq: Analyse U'''  non-local condition}\end{equation}
as a non-local sufficient condition for a rise function to be icpd.
The condition for dcpd $U$ is given when replacing $\le$ by $\ge$.

Now note that if \eqref{eq: Analyse U'''  non-local condition} is
satisfied locally for $\phi=\psi$ the sign of the derivative \[
\frac{\partial}{\partial\psi}\left(3\frac{U''\left(\phi\right)^{2}}{U'\left(\phi\right)}-\frac{U''\left(\psi\right)U''\left(\phi\right)U'\left(\phi\right)}{U'\left(\psi\right)^{2}}\right)=U''\left(\phi\right)U'\left(\phi\right)\left(2\frac{U''\left(\psi\right)^{2}}{U'\left(\psi\right)^{3}}-\frac{U'''\left(\psi\right)}{U'\left(\psi\right)^{2}}\right)\]
is determined by $U''(\phi)$ since the term in brackets on the right
hand side at $\phi=\psi$ is positive using inequality \eqref{eq: Analyse U'''  non-local condition}
and $U'>0$. Hence, if $U$ is concave, a sufficient local condition
for a rise function to be icpd is \[
U''\left(\phi\right)\le0\quad\text{and}\quad U'''\left(\phi\right)\le2\frac{U''\left(\phi\right)^{2}}{U'\left(\phi\right)}\quad\forall\,0\le\phi\le1\]
Conversely a local condition for a convex rise functions to be dcpd
is given by 

\[
U''\left(\phi\right)\ge0\quad\text{and}\quad U'''\left(\phi\right)\ge2\frac{U''\left(\phi\right)^{2}}{U'\left(\phi\right)}\quad\forall\,0\le\phi\le1\:.\]

Different properties of commonly used rise functions are summarized
in Table \ref{cap: Analyse Properties-of-rise-functions}.

\bibliographystyle{abbrv}
\bibliography{arXiv_submit_12092008/arxiv_12092008}

\end{document}